\documentclass[nohyperref]{article}

\usepackage{microtype}
\usepackage{subfigure}
\usepackage{booktabs} % for
\usepackage{graphicx}% http://ctan.org/pkg/graphicx
\usepackage{caption}%professional tables
\usepackage{natbib}
\usepackage{hyperref}

\usepackage[accepted]{icml2022}

\usepackage{amsmath}
\usepackage{amssymb}
\usepackage{mathtools}
\usepackage{amsthm}

\usepackage[capitalize,noabbrev]{cleveref}

\newcommand{\R}{\mathbb{R}}

\newcommand{\N}{\mathbb{N}}

\newcommand{\Lc}{\mathcal{L}}

\newcommand{\Oc}{\mathcal{O}}
\newcommand{\ee}{\varepsilon}
\newcommand{\lo}{\longrightarrow}
\newcommand{\li}{\left}
\newcommand{\re}{\right}

\newcommand{\nol}{||\,}
\newcommand{\nor}{\,||}

\theoremstyle{plain}
\newtheorem{theorem}{Theorem}

\newtheorem{corollary}[theorem]{Corollary}
\theoremstyle{definition}

\theoremstyle{remark}

\newtheorem{example}[theorem]{Example}

\usepackage[textsize=tiny]{todonotes}

\icmltitlerunning{Sobolev Cubatures strengthen the Approximation Power of Physics Informed Neural Nets}

\begin{document}

\twocolumn[
\icmltitle{Replacing Automatic Differentiation by Sobolev Cubatures fastens Physics Informed Neural Nets and strengthens their Approximation Power}



\icmlsetsymbol{equal}{*}

\begin{icmlauthorlist}
\icmlauthor{Juan Esteban Suarez Cardona}{to}
\icmlauthor{Michael Hecht}{to}
\end{icmlauthorlist}

\icmlaffiliation{to}{CASUS - Center for Advanced System Understanding,  Helmholtz-Zentrum Dresden-Rossendorf e.V. (HZDR), G\"{o}rlitz, Germany}

\icmlcorrespondingauthor{Juan Esteban Suarez Cardona}{j.suarez-cardona@hzdr.de}
%\icmlcorrespondingauthor{Dominik Sturm}{d.sturm@hzdr.de}
\icmlcorrespondingauthor{Michael Hecht}{m.hecht@hzdr.de}

\icmlkeywords{Machine Learning, PDEs, Scientific Computing}

\vskip 0.3in
]



\printAffiliationsAndNotice{}  % leave blank if no need to mention equal contribution
% \printAffiliationsAndNotice{} % otherwise use the standard text.

\begin{abstract}
We present a novel class of approximations for variational losses, being applicable for the training of physics-informed neural nets (PINNs). The loss formulation reflects classic Sobolev space theory for partial differential equations and their weak formulations.
The loss computation rests on an extension of \emph{Gauss-Legendre cubatures}, we term \emph{Sobolev cubatures}, replacing \emph{automatic differentiation (A.D.)}. We prove the runtime complexity of training the resulting
Soblev-PINNs (SC-PINNs) to be less than required by PINNs relying on A.D.  On top of one-to-two order of magnitude speed-up the SC-PINNs are demonstrated  to achieve closer solution approximations for prominent forward and inverse PDE problems than established PINNs achieve.
%
% We present a novel class of variational loss approximations, we term \emph{Sobolev cubatures}, being applicable for training of physics informed neural nets (PINNs).  The Sobolev cubatures are given by an extension of the famous \emph{Gauss-Legendre cubatures}, enabling to address
% classic partial differential equation (PDE) problems and their weak formulations. The resulting PINN framework does not rely on \emph{automatic differentiation}, and achieves closer, stable solution approximations for prominent PDE problems compared to established PINNs, as demonstrated herein.
\end{abstract}

\section{Introduction}

Partial differential equations (PDEs) are omnipresent mathematical models governing the dynamics and (physical) laws of complex systems \cite{Jost,brezis2011}.
However,
analytic PDE solutions are rarely known for most of the systems being the center of current research.
Therefore, there is a strong demand on efficient and accurate numerical solvers and simulations.

% Main classic numerical solvers divide into: Finite Elements \cite{ern2004theory}; Finite Differences \cite{LeVeque2007FiniteDM}; Finite Volumes\cite{bernardi1997spectral}; Spectral Methods \cite{bernardi1997spectral} and Particle Methods \cite{li2007meshfree}.

Machine learning methods such as:  Physics-Informed GAN \cite{pmlr-v70-arjovsky17a},
Deep Galerkin Method \cite{Sirignano2018DGMAD}, and Physics Informed Neural Networks (PINNs) \cite{RAISSI2019686}, gain big traction in the scientific computing community.
In contrast to classic solvers, PINNs provide a neural net (NN) surrogate model parametrizing the solution space of the PDEs. PINN-learning is given by minimizing a
variational problem, which is typically formulated in $L^2$-loss terms
\begin{equation}\label{L2}
    \int_{\Omega} \|\hat u(x) - u(x)\|^2 d\Omega \approx \frac{1}{|P|}\sum_{p \in P}\|\hat u(p) - u(p)\|^2
\end{equation}
being approximated by mean square errors (MSE) in random nodes $P$, \cite{Yang2020PhysicsInformedGA},\cite{Long2018PDENetLP}.

However, due to the complexity of the underlying non-linear, non-convex variational problem,  theoretical and computational challenges arise when demanding to guarantee convergence to PDE solutions of high accuracy.

\subsection{Related Work}
The importance of the present computational challenge is manifest in the large number of previous works.
Consequently, an exhaustive overview of the literature cannot be given here.
Instead, we restrict ourselves to mentioning those contributions that directly relate to or inspired our work.
\subsubsection{Classic Numerical Methods}
Main classic numerical solvers divide into: Finite Elements \cite{ern2004theory}; Finite Differences \cite{LeVeque2007FiniteDM}; Finite Volumes\cite{bernardi1997spectral}; Spectral Methods \cite{bernardi1997spectral} and Particle Methods \cite{li2007meshfree}. These class of methods provide solutions with high accuracy, but come with the cost of having limited flexibility with respect to the type of problems they can solve. This includes especially inverse problems, as PDE parameter inference, being a hard challenge for the aforementioned methods. In contrast, the variational formulation defining the PINN training provides the desired flexibility, but comes again with the cost of less reachable accuracy.

We focus on two recent approaches addressing the stability and accuracy of PINNs.
\subsubsection{Variational PINNs (VPINNs)}\label{sec:VPINN}
VPINNs were introduced in  \cite{kharazmi2019variational}, \cite{Kharazmi2020hpVPINNsVP}
formulating variational Sobolev losses for PINN-training. The approach relies on exploiting analytic integration and differentiation formulas of shallow neural networks with specified  activation functions.The method is extended by using quadratures and automatic differentiation for computing the losses and complemented by a domain decomposition approach.
% This method
% and hence is limited in the class of non-linear functions that could approximate. The authors extend the approach by using quadratures for computing the integrals of the losses and automatic differentiation for the gradients, which can be applied to more general architectures. The authors include in the second paper a further improvement of the method by using domain decomposition when computing the losses. We benchmark VPINNs against the Sobolev cubatures by solving the 1D Poisson equation with two set of parameters.

\subsubsection{Inverse Dirichlet Loss Balancing}\label{sec:ID}
With the Inverse Dirichlet method \cite{maddu2021} the numerical
stability of PINNS was increased by dynamically balancing the
occurring gradient amplitudes, which if unbalanced
cause numerical stiffness phenomena  \cite{wang2021understanding}. The PINN formulation, however, rests on classic MSE losses.
% , a big difference on the rates of the training, leeds to high eigenvalues on the Hessian matrix of the total loss. The I.D. method consist on a dynamic scaling of the PDE loss, by computing the ratio between the standard deviation of the gradient of the PDE loss and the gradient of the boundary loss or the reconstruction loss. In this way, the variance of both gradients has the same order, which it is supposed to prevent unwanted effects such as vanishing gradients.
% We benchmark our methods with the Sobolev Cubatures against this method, for the 2D Poisson equation and the 2D time independent Quantum Harmonic Oscillator.

\subsection{Contribution}
We provide the notion of Sobolev cubatures providing
close approximations of the corresponding Sobolev norms, extending the $L^2$-integrals in Eq~\eqref{L2}.
While classic weak PDE formulations rest on the Sobolev spaces,
our approach enables to
mathematically prove the resulting PINNs to converge to strong solutions of sufficiently regular posed PDE problems. Moreover,
the Sobolev cubatures enable to replace  \emph{automatic differentiation} (A.D.) \cite{baydin2018} by \emph{polynomial differentiation} (P.D.).
Being mostly the bottleneck, for both, runtime and accuracy performance, replacing A.D. by P.D. does not only result in more efficient PINNs, but allows to reach
much closer solution approximations (several magnitudes), as demonstrated herein.

\section{Sobolev Cubatures}
Closer approximations of $L^2$-integrals than reachable by prominent MSE approaches can be derived by applying
classic \emph{Gauss-Legendre cubatures} \citep[see e.g.][]{stroud,stroud2,trefethen2017,trefethen2018,Trefethen2019} or even further extension to what we call \emph{Sobolev cubatures}, presented herein. The notions of this section follow classic concepts of spectral methods \cite{canuto2007spectral,Trefethen2019}. We start by introducing fixing the notation used in the article.

\subsection{Notation}

We consider neural nets (NNs) $\hat u_\theta(\cdot)$ with $m_1,m_2$-dimensional input/output domain $m_1,m_2 \in \N$ of fixed architecture $\Xi_{m_1,m_2}$ (specifying number and depth of the hidden layers, with continuous piecewise smooth activation functions $\sigma(x)$, e.g.  ReLU or ELU. Further, $\Upsilon_{\Xi_{m_1,m_2}}$ denotes the parameter space of the weights (and bias) $\theta =(v,b) \in W=V\times B \subseteq \R^K$, $K \in \N$,
see e.g. \cite{martin2009,goodfellow2016}.

% Let $m,n \in \N$, $p\geq 1$. We denote by  $e_1=(1,0,\dots,0), \dots,e_m = (0,\dots,0,1) \in \R^m$ the standard basis, by $\|\cdot\|$ the euclidean norm on $\R^m$, and by $\|M\|_p$ the $l_p$-norm
% of a matrix $M\in \R^{m\times m}$. Further, $A_{m,n,p} \subseteq \N^m$ denotes all multi-indices $\alpha =(\alpha_1,\dots,\alpha_m)\in \N^m$ with $\|\alpha\|_p = (\alpha_1^p+\dots+\alpha_m^p)^{1/p}\leq n$.
% We order a finite set $A\subseteq \N^m$, $m \in \N$, of multi-indices with respect to the lexicographical order $\leq_L$ on $\N^m$ proceeding from the last entry to the first, e.g.,
% $(5,3,1)\leq_L (1,0,3) \leq_L (1,1,3)$.
% We call $A$
% \emph{downward closed} or \emph{complete} if and only if there is no $\beta = (b_1,\dots,b_m) \in \N^m \setminus A$
% with $b_i \leq a_i$, $ \forall \,  i=1,\dots,m$ for some $\alpha = (a_1,\dots,a_m) \in A$. This follows the terminology of ``downward closed sets'' introduced by \citet{cohen3}. The sets $A_{m,n,p}$ are complete for all $m,n\in \N$, $p\geq 1$ and induce the generalized notion of \emph{polynomial $l_p$-degree} as follows:

Throughout this article we denote with $\Omega=[-1,1]^m$ the $m$-dimensional \emph{standard hypercube} and with $\|x\|_p = (\sum_{i=1}^m |x_i|^p)^{1/p}$, $x=(x_1,\ldots,x_m) \in \R^m$, $1 \leq p < \infty$, $\|x\|_{\infty} = \max_{1\leq i\leq m} |x_i|$ the $l_p$-norm.
We recommend \cite{Adams2003,brezis2011} for an excellent overview on functional analysis and the Sobolev spaces
$$ H^k(\Omega,\R) = \li\{ f \in L^2(\Omega,\R) : D^\alpha f \in L^2(\Omega,\R)\re\}\,, k\in \N$$
given by all $L^2$-integrable functions $f : \Omega \lo \R$ with existing  $L^2$-integrable weak derivatives $D^\alpha f= \partial^{\alpha_1}_{x_1}\ldots\partial^{\alpha_m}_{x_m}f$ up to order $k$. In fact,
$H^k(\Omega,\R)$ is a Hilbert space with inner product
$$
    \li<f,g\re>_{H^k(\Omega)} =  \sum_{0 \leq \|\alpha\|_1 \leq k}\li<D^\alpha f,D^\alpha g\re>_{L^2(\Omega)}
$$
and norm $\|f\|_{H^k(\Omega)}^2 = \li<f,f\re>_{H^k(\Omega)}$, where
$H^0(\Omega,\R) = L^2(\Omega,\R)$, with $<f,g>_{L^2(\Omega)} = \int \limits_{\Omega}f\cdot g \, d\Omega$.

From the \emph{Trace Theorem}~\citep{Adams2003}, we find furthermore that whenever $H \subseteq \R^m$ is a hyperplane of co-dimension $1$, then the induced restriction
\begin{equation}\label{eq:trace}
  \varrho: H^k(\Omega,\R) \lo H^{k-1/2}(\Omega \cap H,\R)
\end{equation}
is continuous, i.e., $\nol f_{| \Omega \cap H} \nor_{H^{k-1/2}(\Omega\cap H)} \leq d\nol f \nor_{H^k(\Omega)}$  for some $d=d(m,\Omega) \in \R^+$.

Further, $C^k(\Omega,\R)$, $k \in \N \cup\{\infty\}$ denote the
\emph{Banach spaces}
of $k$-times continuously differentiable functions with norm $\|f\|_{C^k(\Omega)} = \sum_{i=0}^k \sup_{x \in \Omega,\|\alpha\|_1 =i}|D^\alpha f(x)|$.

$\Pi_{m,n} = \mathrm{span}\{x^\alpha\}_{\|\alpha\|_\infty \leq n}$  denotes the $\R$-\emph{vector space of all real polynomials} in $m$ variables spanned by all monomials
$x^{\alpha} = \prod_{i=1}^mx_i^{\alpha_i}$ of \emph{maximum degree} $n \in \N$ and
$A_{m,n}=\{\alpha \in \N^m : \|\alpha\|_\infty \leq n\}$ the corresponding multi-index set with
$|A_{m,n}| = (n+1)^m$. We order $A_{m,n}$ with respect to the \emph{lexicographic order} $\preceq$.
Let $D=(d_{i,j})_{1 \leq i,j \leq |A_{m,n}|} \in \R^{|A_{m,n}|\times |A_{m,n}|}$ be a matrix.
We slightly abuse notation by writing
\begin{equation}
 D = (d_{\alpha,\beta})_{\alpha,\beta \in A_{m,n}}  \,,
\end{equation}
with $d_{\alpha,\beta}$ being the $\alpha$-th, $\beta$-th entry of D.
% In particular, $e_\alpha = e_j \in \R^{|A_{m,n}|}$ denotes the $j$-th standard basis vector.
% Finally, we use the standard \emph{Landau symbols} $f \in \Oc(g)  \Leftrightarrow \lim \sup_{x\rightarrow \infty} \frac{|f(x)|}{|g(x)|} \leq \infty$, $f \in o(g)  \Leftrightarrow \lim_{x\rightarrow \infty} \frac{|f(x)|}{|g(x)|} =0$.

\subsection{Orthogonal Polynomials}

Let $m,n\in \N$ we consider the $m$-dimensional Legendre grids $P_{m,n} = \oplus_{i=1}^m \mathrm{Leg_n} \subseteq \Omega$,
where $\mathrm{Leg}_n=\{p_0,\ldots,p_n\}$ are the $n+1$ \emph{Legendre nodes} given by the roots of the \emph{Legendre polynomials} of degree $n+2$
% \citep[see e.g.][]{gautschi2011}
and denote $p_{\alpha} = (p_{\alpha_1}, \ldots, p_{\alpha_m}) \in P_{m,n}$. It is a classic fact \citep{stroud,stroud2,trefethen2017,Trefethen2019} that the Lagrange polynomials $L_{\alpha}\in \Pi_{m,n}$, $\alpha \in A_{m,n}$  given by
\begin{equation}\label{eq:Lag}
    L_{\alpha} = \prod_{i=1}^ml_{\alpha_i,i}\,, \quad l_{j,i} = \prod_{h \not = j, h=0}^n \frac{x_i -p_h}{p_j-p_h}\,, \end{equation}
$p_h \in \mathrm{Leg}_n$ satisfy $L_{\alpha}(p_\beta)=\delta_{\alpha,\beta}$,
$\forall \alpha,\beta \in A_{m,n}$ and
yield an orthogonal $L^2$-basis of $\Pi_{m,n}$, i.e.,
$$\li<L_{\alpha},L_{\beta}\re>_{L^2(\Omega)}=\int \limits_{\Omega} L_{\alpha}(x)L_{\beta}(x)d\Omega = w_{\alpha} \delta_{\alpha,\beta}\,,$$
$\forall \, \alpha,\beta \in A_{m,n}$, where $\delta_{\cdot,\cdot}$ denotes the \emph{Kronecker delta} and
\begin{equation}\label{eq:GLW}
    w_{\alpha} = \|L_{\alpha}\|^2_{L^2(\Omega)}
\end{equation} the efficiently computable \emph{Gauss-Legendre cubature weight} \cite{stroud,Trefethen2019}.
Consequently, for any polynomial $Q \in \Pi_{m,2n+1}$ of degree $2n+1$ the following cubature rule applies:
\begin{equation}\label{eq:Gauss}
    \int \limits_{\Omega} Q(x)d\Omega = \sum_{\alpha \in A_{m,n}} w_{\alpha}Q(p_{\alpha})\,.
\end{equation}
Summarizing: Polynomials of degree $2n+1$ can be (numerically) integrated exactly when sampled on the Legendre grid $P_{m,n}$ of order $n+1$. Thanks to $|P_{m,n}| = (n+1)^m \ll (2n+1)^m$ this makes \emph{Gauss-Legendre integration} a very powerful integration scheme yielding
\begin{equation}\label{eq:L2}
    \|Q\|^2_{L^2(\Omega)} = \int \limits_{\Omega_m} Q(x)^2d\Omega_m = \sum_{\alpha \in A_{m,n}} Q(p_\alpha)^2 w_\alpha \,.
\end{equation}
for all $Q \in \Pi_{m,n}$.
\subsection{Differential Operators on $\Pi_{m,n}$}
Based on Eq.~\eqref{eq:Lag} we derive exact matrix representations of differential operators on the polynomial spaces $\Pi_{m,n}$ allowing to extend Eq.~\eqref{eq:L2}  to close approximations of the Sobolev norms for general Sobolev functions $f \in H^k(\Omega,\R)$, $k \in \N$.

For $L_\alpha \in \Pi_{m,n}$ from Eq.~\eqref{eq:Lag} and $ 1\leq i \leq m$ the computation of values $d_{\alpha,\beta} = \partial_{x_i} L_\alpha (p_{\beta})$, $p_\beta\in P_{m,n}$, $\forall\,\beta \in A_{m,n}$ yields the Lagrange expansion
\begin{equation}
    \partial_{x_i} L_\alpha (x) = \sum_{\beta \in A_{m,n}} d_{\alpha,\beta}L_{\beta}(x) \,.
\end{equation}
Consequently, the matrix
\begin{equation}\label{eq:DI}
    D_i = (d_{\alpha,\beta})_{\alpha,\beta \in A_{m,n}} \in \R^{|A_{m,n}|\times |A_{m,n}|}\,,
\end{equation}
represents the finite dimensional truncation of the differential operator
$\partial_{x_i} : C^1(\Omega,\R) \lo C^0(\Omega,\R)$ to the polynomial space $\Pi_{m,n}$ and for $\beta \in \N^m$ we set
\begin{equation}\label{eq:Dbeta}
    D_{\beta} = \prod_{j=1}^m D_{\beta_i}
\end{equation}
to be the approximation of $\partial_\beta:=\partial_{x_1}^{\beta_1}\ldots\partial_{x_m}^{\beta_m}$.
As a direct consequence of Eq.~\eqref{eq:L2} the following statement applies:
%Is the equation for the higher order diff. operator matrix correct?
\begin{theorem}[Sobolev cubatures] \label{theo:SOB} Let $m,n\in \N$, $A_{m,n} \subseteq \N^m$,
$P_{m,n} = \{p_\alpha : \alpha \in A_{m,n}\}$,  be the Legendre grid,  $W_{m,n} = \mathop{diag}(w_{\alpha})_{\alpha \in A_{m,n}}$ be the Gauss-Legendre cubature weight matrix, and $F_{m,n} = (f(p_{\alpha}))_{\alpha \in A_{m,n}}\in \R^{|A_{m,n}|}$. Let
\begin{equation}\label{eq:SCUB}
    W_{m,n,k} = \sum_{\beta \in \N^m, \|\beta\|_1\leq k}  D_{\beta}^T W_{m,n} D_\beta\,,
\end{equation}
with $D_\beta$ from Eq.~\eqref{eq:Dbeta}
be the Sobolev cubature matrix then
  \begin{align}\label{eq:SOBapp}
      \|f\|_{H^k(\Omega)}^2 &\approx  F^T_{m,n} W_{m,n,k}F_{m,n}\,,
    %   \sum_{\beta \in A_{m,n}} f(p_{\beta})^2\mu_{\alpha,\beta}
  \end{align}
is an exact approximation whenever $f \in \Pi_{m,n}$.
\end{theorem}

We conclude that Theorem~\ref{theo:SOB}  enables to  control the uniform
distance $\|f -g\|_{C^0(\Omega)}$ on the whole domain $\Omega$.
\begin{corollary}\label{cor:SOB} Let the assumptions of Theorem~\ref{theo:SOB} be fulfilled, and  $f,g \in H^{k}(\Omega,\R)$, $k>m/2$ be two Sobolev functions.
Assume there is $n \in \N$ (large enough)
such that the residuum $f - g \in \Pi_{m,n}$ is given by a polynomial with
$$f(p_\alpha) - g(p_\alpha) = 0 \,, \quad \forall\, p_\alpha \in P_{m,n}\,. $$
Then $f(x) -g(x) = 0$ for all $x \in \Omega =[-1,1]^m$.
\end{corollary}

\begin{proof} Due to Theorem~\ref{theo:SOB} we deduce
$\|f-g\|_{H^k(\Omega)} = 0$. While the Sobolev Embedding Theorem \cite{Adams2003} yields the continuous inclusion $H^k(\Omega,\R) \subseteq C^0(\Omega,\R)$, consequently, we  realize that $\|f-g\|_{C^0(\Omega)}=0$, proving the claimed identity.
\end{proof}

% \begin{remark}
% Note, that apart from the constraint $f - g \in \Pi_{m,n}$ of the residuum of $f,g$ to be a polynomial, $f,g \in H^k(\Omega,\R)$, $k>m/2$ can be
% general non-polynomial, Sobolev functions.
% \end{remark}

In light of the provided perspectives, we propose the following formalizations of classic PDE problems.

\section{Strong and weak PDE formulations}\label{Poisson_formulations}

We follow \cite{Jost,brezis2011} to restate classic (weak) PDE formulations and their Sobolev cubature approximations. For the sake of simplicity, we focus on the example of the classic \emph{Poisson equation}.

\subsection{Poisson equation}
For $f \in C^0(\Omega,\R)$ the \emph{strong Poisson problem} with Dirichlet boundary condition $g \in C^0(\partial \Omega,\R)$ seeks for solutions $u\in C^2(\Omega,\mathbb{R})$ fulfilling:
\begin{equation}\label{eq:PP}
   \li\{ \begin{array}{rll}
       -\Delta u(x) -f(x) &= 0  &,  \forall x\in\Omega  \\
         u(x)  -g(x)      &= 0  &,  \forall x\in\partial\Omega\,.
\end{array}\re.
\end{equation}

We can weaken the initial regularity assumptions by  demanding $u\in H^2(\Omega,\mathbb{R})$ to satisfy the PDE in the integral sense, yielding the \emph{strong variational formulation}:
\begin{equation}\label{eq:strong_var}
    \li\{\begin{array}{rll}
        \int \limits_{\Omega}(-\Delta u + f) \phi\, d\Omega &= 0   &,
        \forall \phi \in \Gamma(\Omega,\mathbb{R}) \\
        \int \limits_{\partial\Omega}(u - g)\phi \,dS &= 0  &,
        \forall \phi \in \Gamma(\partial\Omega,\mathbb{R})\,,
    \end{array}\re.
\end{equation}
where $\Gamma(\Omega,\R) = \{\phi \in C^{\infty}(\Omega,\R) : \|\phi\|_{C^0(\Omega)} \leq 1\}$ denotes the space of \emph{test functions}.
We can weaken the regularity assumptions even more by  imposing $u\in{H^1}(\Omega,\mathbb{R})$ to satisfy the following \emph{weak variational formulation}:
\begin{equation}\label{eq:weak_var}
    \li\{\begin{array}{rl}
        \int\limits_{\Omega}(\nabla u \cdot \nabla\phi + f\phi) d\Omega -
        \int \limits_{\partial\Omega}(\nabla u \phi)\eta dS &= 0 \\
\int \limits_{\partial\Omega}(u - g)\phi dS &= 0\,,
    \end{array}\re.
\end{equation}
where we applied integration by parts and $\eta$ denotes the (piecewise smooth) normal field of $\partial \Omega$.

% The Sobolev cubatures, derived in Theorem~\ref{theo:SOB},
% enabling practical PINN implementations ensuring that  Eq.~\eqref{eq:strong_var},~\eqref{eq:weak_var} are satisfied for all smooth (polynomial) test functions $\phi \in \Pi_{m,n} \subseteq C^\infty(\Omega,\R)$ whenever the corresponding residuum vanishes.
\subsection{Residual loss in terms of Sobolev cubatures}
We translate the introduced PDE formulations into \emph{variational optimization problems} demanded to be minimized by the PINNs framework. In addition to Sobolev cubature matrix $W_{m,n,k}$ from Eq.~\eqref{eq:SCUB}
the matrices
\begin{align}
    U_{m,n,k} &= \sum_{\beta \in \N^m, \|\beta\|_1 \leq k} D_\beta^TW_{m,n}^2D_\beta\,,
    \label{eq:UMNK}\\
    V_{m-1,n,k-1/2} &= \sum_{\beta \in \N^{m-1}, \|\beta\|_1 \leq k-1/2} D_\beta^TW_{m-1,n}^2D_\beta \,, \nonumber
\end{align}
relying on Theorem~\ref{theo:SOB}, are the key ingredient in this regard.

The {\bf strong residual loss} $\Lc_{\mathrm{strong}} : C^2(\Omega,\R) \lo \R$, relies on the residuals
\begin{align*}
\Lc_{\mathrm{strong}}(u) &= r_{\mathrm{strong},0}(u) + s_{\mathrm{strong},0}(u)\\
   &= \|\Delta u +f \|^2_{L^2(\Omega)} + \| u_{| \partial \Omega} -g \|^2_{L^2(\partial\Omega)}\,,
\end{align*}
which we extend to  $\Lc_{\mathrm{strong},k}$, $k\geq 1/2$ with
\begin{align}
\Lc_{\mathrm{strong},k}(u) &= r_{\mathrm{strong},k}(u) + s_{\mathrm{strong},k}(u)\label{eq:strong}\\
   &= \|\Delta u +f \|^2_{H^k(\Omega)} + \| u_{| \partial \Omega} -g \|^2_{H^{k-1/2}(\partial\Omega)}\,, \nonumber
\end{align}
where the $H^{k-1/2}$-metric of the second residual $s_{\mathrm{strong},k}$ reflects the \emph{Trace Theorem} (Eq.~\eqref{eq:trace}).
We propose to approximate $\Lc_{\mathrm{strong},k}(u)  \approx r_{\mathrm{strong},n,k}(u) + s_{\mathrm{strong},n,k}(u)$ by the following Sobolev cubatures:
\begin{align}
r_{\mathrm{strong},n,k}(u) &=R_{m,n}^T W_{m,n,k} R_{m,n}\,,  \label{eq:strong_loss}  \\
s_{\mathrm{strong},n,k}(u) &= S_{m-1,n}^{T} V_{m-1,n,k-1/2} S_{m-1,n}\,.
\end{align}

Thereby, $V_{m-1,n,k-1/2} = W_{m-1,n}$ for $0\leq k <1/2$, and
\begin{equation}\label{eq:R}
R_{m,n}  = -D_{(2,\ldots,2)}(\hat u (p_\alpha,w))_{\alpha \in A_{m,n}} - F_{m,n}\,,
\end{equation}
where $F_{m,n} = (f(p_{\alpha}))_{\alpha \in A_{m,n}}$ and
$D_{(2,\ldots,2)} = D_1^2 + \ldots + D_m^2 \approx \Delta$ denotes the polynomial approximation of the Laplacian accordingly to Eq.~\eqref{eq:DI}. Thus,
$R_{m,n} \in \R^{|A_{m,n}|}$ yields an approximation of
$\big(\Delta u(p_{\alpha}) + f(p_{\alpha})\big)_{\alpha \in A_{m,n}}$ by replacing \emph{automatic differentiation} \cite{baydin2018} with  \emph{polynomial differentiation}.

Moreover,  by summing the residual values
$S_{m-1,n}^{\pm j} = (u(p^{\pm j}_{\alpha},w) - g(p^{\pm j}_{\alpha}))_{\alpha \in A_{m-1,n}}$
over each face $\partial\Omega_{j}^{\pm} =\{x \in \Omega : x_j = \pm 1\}$ of $\Omega$ we denote the boundary residual as
\begin{equation}
S_{m-1,n} = \sum_{ j=0}^{m}S_{m-1,n}^{\pm j}  \in \R^{|A_{m-1,n}|}\,.\label{eq:SMN}
\end{equation}

The {\bf strong variational loss} $\Lc^{\mathrm{var}}_{\mathrm{strong},k} : H^{2+k}(\Omega,\R) \lo \R$ is given by
$\Lc^{\mathrm{var}}_{\mathrm{strong},k}(u) = r^{\mathrm{var}}_{\mathrm{strong},k}(u) +s^{\mathrm{var}}_{\mathrm{strong},k}(u)$ with
\begin{align}
 r^{\mathrm{var}}_{\mathrm{strong},k}(u) &:=
\sup_{\phi \in \Gamma(\Omega,\R)} \li <-\Delta u-f,\phi \re>^2_{H^k(\Omega)} \label{eq:strongvar} \\
s^{\mathrm{var}}_{\mathrm{strong},k}(u)&= \sup_{\phi \in \Gamma(\partial\Omega,\R)} \li <u-g,\phi \re>^2_{H^{k-1/2}(\Omega)}\,, \nonumber
\end{align}
where the $H^{k-1/2}$-metric of the second residual reflects again the \emph{Trace Theorem} (Eq.~\eqref{eq:trace}). Replacing the test functions $\phi$  with
the Lagrange basis $L_{\alpha} \in \Pi_{m,n}$, $\alpha \in A_{m,n}$ for the Legendre nodes $p_{\alpha} \in P_{m,n}$ from Eq.~\eqref{eq:Lag}
yields the Sobolev cubature approximation
$\Lc^{\mathrm{var}}_{\mathrm{strong},k}(u)  \approx \Lc^{\mathrm{var}}_{\mathrm{strong},n,k}(u)$:
\begin{align*}
r^{\mathrm{var}}_{\mathrm{strong},n,k}(u)&= R_{m,n}^T U_{m,n,k} R_{m,n}\,,\\
    s^{\mathrm{var}}_{\mathrm{strong},n,k}(u) &= S_{m-1,n}^T V_{m-1,n,k-1/2} S_{m-1,n}\,,
\end{align*}
where $R_{m,n}$, $S_{m-1,n}$ are as in  Eq.~\eqref{eq:R},\eqref{eq:SMN}.

%Check Sup on the weak formulation%
The {\bf weak variational loss} $\Lc^{\mathrm{var}}_{\mathrm{weak},k}: H^{1+k}(\Omega,\R) \lo \R$ given by
$\Lc^{\mathrm{var}}_{\mathrm{weak},k}= r^{\mathrm{var}}_{\mathrm{weak},k}(u) + s^{\mathrm{var}}_{\mathrm{strong},k}(u)$ only differs in the first residual
$r^{\mathrm{var}}_{\mathrm{weak},k}(u)$ from $\Lc^{\mathrm{var}}_{\mathrm{strong},k}$  by
\begin{align*}
r^{\mathrm{var}}_{\mathrm{weak},k}(u) & = \sup_{\phi \in \Gamma(\Omega,\R)}\Big( \li <\nabla u,\nabla\phi\re>_{H^k(\Omega)} +\li <f,\phi\re>_{H^k(\Omega)}\\
&- \li <\nabla u\phi, \eta\re>_{H^{k-1/2}(\partial\Omega)}\Big)\,.
\end{align*}
Yielding the approximates $\Lc^{\mathrm{var}}_{\mathrm{weak},k}(u) \approx \Lc^{\mathrm{var}}_{\mathrm{weak},n,k}(u)$ due to:
\begin{align*}
\sup_{\phi \in \Gamma(\Omega,\R)} \li <\nabla u,\nabla\phi\re>_{H^k(\Omega)} &\approx  H_{m,n}^T U_{m,n,k} H_{m,n}\\
\sup_{\phi \in \Gamma(\Omega,\R)}\li <f,\phi\re>_{H^k(\Omega)}
&\approx F_{m,n}^T U_{m,n,k} F_{m,n} \\
\sup_{\phi \in \Gamma(\Omega,\R)}\li <\nabla u\phi, \eta\re>_{H^{k-1/2}(\partial\Omega)}
&\approx G_{m-1,n}^T V_{m-1,n,l} G_{m-1,n},
\end{align*}
where $H_{m,n}= \sum_{i=1}^mD_i^TD_i (u(p_{\alpha}))_{\alpha \in A_{m,n}}$, $F_{m,n}= (f(p_{\alpha}))_{\alpha \in A_{m,n}}$,
$G_{m-1,n} = e_{\alpha}^T \sum_{i=1}^m D_i  (u(p_{\alpha}))_{\alpha \in A_{m,n}}$ and $U_{m,n,k}$, $V_{m-1,n,l}$, $l =k-1/2$ are the Sobolev cubature matrices from Eq.~\eqref{eq:UMNK}.

Summarizing, the given notions allow to extend Corollary~\ref{cor:SOB} in order to state the following result.
\begin{theorem}[Strong solution approximation]\label{theo:ss} Let $m \in \N$, $u \in C^2(\Omega,\R)$, $u \in H^{l+k}(\Omega, \R)$, $k\geq 0$, $l=1,2$ be a regular $m$-variate (Sobolev) function.
\begin{enumerate}
\item[i)] If the losses vanish, $\Lc_{\mathrm{strong},k}(u)=0$, $\Lc^{\mathrm{var}}_{\mathrm{strong},k}(u)=0$, or $\Lc^{\mathrm{var}}_{\mathrm{weak},k}(u) =0$ then $u$
is a strong solution of the Poisson equation, respectively, i.e, $u \in C^2(\Omega,\R)$ and solves the strong Poisson problem.
\item[ii)] Denoting with $\Lc_{\mathrm{strong},n,k}$, $\Lc^{\mathrm{var}}_{\mathrm{strong},n,k}$, $\Lc^{\mathrm{var}}_{\mathrm{weak},n,k}$ the loss approximations of degree $n \in \N$ as given above. Then for all $\ee>0$ there is $n =n(\ee)$ and $u' \in \Pi_{m,n}$ with $\|u -u'\|_{C^{2}(\Omega)}<\ee$, $\|u -u'\|_{H^{l+k}(\Omega)}<\ee$ such that
$\Lc_{\mathrm{strong},n}(u') = \Lc_{\mathrm{strong}}(u')$,
$\Lc^{\mathrm{var}}_{\mathrm{strong},n,k}(u') = \Lc^{\mathrm{var}}_{\mathrm{strong},k}(u')$,
$ \Lc^{\mathrm{var}}_{\mathrm{weak},n,k}(u') = \Lc^{\mathrm{var}}_{\mathrm{weak},k}(u')$.
\end{enumerate}
\end{theorem}
\begin{proof}Statement $i)$ reflects classic PDE theory \cite{Jost,brezis2011}. $ii)$ follows from
$H^{l+k}(\Omega,\R) \subseteq C^2(\Omega,\R)$ and
the Stone-Weierstrass theorem \cite{weier1,weier2}, stating that any continuous function $f \in C^0(\Omega,\R)\supseteq C^2(\Omega,\R)$ can be uniformly approximated by polynomials.
By a bootstrapping argument that implies that polynomials are dense in all $C^r(\Omega,\R)$, $r \in \N\cup\{\infty\}$ with respect to the corresponding norms $\|\cdot\|_{C^r(\Omega)}$.
While $C^\infty(\Omega,\R) \subseteq H^{l+k}(\Omega,\R)$ is dense \cite{Adams2003},
polynomial approximations $\|u -u'\|_{H^{l+k}(\Omega)}<\ee$ exist.
Due to Theorem~\ref{theo:SOB} the Sobolev cubature losses $\Lc_{\mathrm{strong},n}(u')$,       $\Lc^{\mathrm{var}}_{\mathrm{strong},n,k}(u')$,
$\Lc^{\mathrm{var}}_{\mathrm{weak},n,k}(u')$
are exact for $u' \in \Pi_{m,n}$ proving the statement.
\end{proof}

We conclude, that apart from the weaker regularity assumptions on the solution $u : \Omega \lo \R$ required by the variational losses, Theorem~\ref{theo:ss}  theoretically guarantees approximations of strong solutions whenever minimization (by PINN training)
reaches $\Lc_{\mathrm{strong},n,k}(u'),    \Lc^{\mathrm{var}}_{\mathrm{strong},n,k}(u'),
\Lc^{\mathrm{var}}_{\mathrm{weak},n,k}(u')\approx 0$ sufficient small losses for $n \in \N$ large enough.

\section{Gradient Flow of PINN Training}
For a given PINN $u =\hat u(\cdot,w) \in \Xi_{m,1}$, of fixed architecture,
(optimally) adjusted weights $w \in \Upsilon_{\Xi_{2,1}}$ are demanded,
minimizing the loss. NN training is realized as a gradient descent, given by solving the \emph{gradient flow ODE}
\begin{align}
 \partial_t w  &= - \delta_w \Lc(\hat u(\cdot,w(t))     w_0\,,\\
  &= - \nabla \Lc(\hat u(\cdot,w(t))\cdot \delta_w  \hat u(\cdot,w(t))w_0        \quad  \,, w(0) = w_0\,,\nonumber
\end{align}
where $w_0$ is given by the NN initialization and $\delta_w$ denotes the variation in the weights.

For the proposes Sobolev loss approximations
$\Lc_{\mathrm{strong},n}(u'),\Lc^{\mathrm{var}}_{\mathrm{strong},n}(u'),\Lc^{\mathrm{var}}_{\mathrm{weak},n}(u')$
the gradients simplify due to the identities:
\begin{align}
   \nabla_w(r_{\mathrm{strong},n,k}(u)) & = 2R_{m,n}^T W_{m,n,k} \nonumber\\
   \nabla_w(r^{\mathrm{var}}_{\mathrm{strong},n,k}(u)) & = 2R_{m,n}^T U_{m,n,k}\nonumber\\
     \nabla_w(r^{\mathrm{var}}_{\mathrm{weak},n,k}(u)) & = 2(H_{m,n}^T + F_{m,n}^T)U_{m,n,k}\nonumber\\
     &+ 2(G_{m,n}^T )V_{m-1,n,k-1/2}\nonumber\\
   \nabla_w(s_{\mathrm{strong},n}(u))  & = 2R_{m,n}^T W_{m-1,n} \nonumber\\
   \nabla_w(s^{\mathrm{var}}_{\mathrm{strong},n,k}(u)) & = 2R_{m,n}^T V_{m-1,n,k-1/2}\,, \label{eq:GRAD}
\end{align}
where $W_{m,n,k},U_{m,n,k}$, $V_{m-1,n,k-1/2}$ are Sobolev cubature matrices from  Eq.~\eqref{eq:SCUB},\eqref{eq:UMNK}.

We further investigate the Sobolev cubature properties in comparison to A.D. based PINNs below.

\section{Polynomial Differentiation vs Automatic Differentiation}\label{sec:ADPD}
One of the main features of using the Sobolev cubatures, compared to the MSE loss, is that we can replace the problem of computing the derivatives of the PINN-surrogate, by computing them directly in the cubature. Below we present the complexity analysis for both, the normal PINN with \emph{Automatic Differentiation (A.D.)} and the SC-PINN with \emph{Polynomial Differentiation (P.D.).}
 \begin{theorem}\label{theo:complex}
%  Let $\hat u(\cdot,w)$ be a neural net of fixed architecture $\Xi_{m_1,m_2}$ consisting of $l\in N$ hidden layers (specifying number and depth of the hidden layers, with continuous piecewise smooth activation functions $\sigma(x)$, e.g.  ReLU or ELU, with input dimension $m_1$ and output dimension $m_2$. Further, $\Upsilon_{\Xi_{m_1,m_2}}$ denotes the parameter space of the weights (and bias) $w =(v,b) \in W=V\times B \subseteq \R^K$, $K \in \N$,
% see e.g. \cite{martin2009,goodfellow2016}.
For a given deep Neural Network $\hat{u}_\theta : \Omega \rightarrow \R$, with architecture $\xi_{m,1}$ consisting of $l$ hidden layers and $q$ neurons per layer, the complexity per epoch for computing the $k-$th derivative ($\partial_x^k \hat{u}_\theta$) in $s \in \N$ training points is given by
\begin{itemize}
    \item[i)]  $\Oc(2^{k-1}lsq^2)$ for a PINN resting on A.D., i.e. it scales exponentially with the order of the derivative.
    \item[ii)] $\Oc(\max\{s^2,lsq^2\})$ for the SC-PINN using P.D.
\end{itemize}
\end{theorem}
\begin{proof}
For proving $i)$ we use the fact that a single backward pass required for computing the derivatives, has the same complexity as a forward pass, which is $O(lsq^2)$ due to \cite{JMLR:v18:17-468}. thus, for computing the first derivative at $s$ points we need $O(lsq^2)$ operations. Due to the chain rule, computation of the second order derivatives causes the size of the dependency graph of the A.D. to double. By recursion of this fact the factor $2^{k-1}$ appears as claimed.

For proving $ii)$ we recall that the SC-PINN computes the $k-$th derivative by applying the pre-computed differential matrix $\mathbb{D}^k:=\Pi_{l=1}^k\mathbb{D}$. Hence, for each epoch the costs of a matrix vector multiplication $\mathbb{D}^k\hat{u}_\theta$ apply at each cubature (Legendre) node $p \in P_{m,n}$, with $s =|P_{m,n}|$.  The product has $O(s^2)$ complexity per epoch and the evaluation $\hat{u}_\theta(p)$, $\forall p \in P_{m,n}$ has $O(lsq^2)$ complexity, yielding the result.
\end{proof}
\begin{example}\label{ex:AD}
We consider a NN  with architecture $\Xi_{1,1}= \{1,50,50,50,50,1\}$ given by four hidden layers of length $50$. For $s=200$ training points computation of a $4$-th order derivative computed with A.D. has a theoretical computational cost of $1.6\cdot 10^7$ operations while P.D. requires $2\cdot 10^{6}$ operations. In fact, the predicted
factor-$10$-speed-up is achieved for the experiment in Section~\ref{sec:exAD}.
\end{example}
In addition to the derivative computation complexity, the SC-PINN formulation exploits the approximation power of the Gauss-Legendre cubature, as it becomes observable in the numerical experiments, Section~\ref{sec:NUM}.

Here, we want to point out the following insight:
We consider the $m$-dimensional integral $I[f]:=\int_\Omega f d\Omega$ of a
$k$-times differentiable function once approximated by the Gauss-Legende cubature of degree $n$, $I_n^g[f^2] := \sum_{\alpha \in A_{m,n}}f^2(p_\alpha)\omega_\alpha$ and once by the Monte Carlo approximation $I_n^M[f]:= \frac{1}{N}\sum_{i\in I}f^2(x_i)$, with $K\subset \mathbb{R}^m$ a set of points of size $|K|=|A_{m,n}|$. Due to \cite{trefethen2017} the approximations rates scale as:
\begin{align*}
    |I[f]-I^g_n[f]|&=\Oc\biggr(\frac{1}{k(n-k)^k}\biggr)\,,\\
    |I[f]-I^M_n[f]|&=\Oc\biggr(\frac{1}{\sqrt{K}}\biggr)\,, \quad  K =(n+1)^m\,.
\end{align*}
Hence, for regular functions, $k \gg m$, achieving a similar accuracy with the SC-PINNs requires less points compared to applying PINNs with MSE-losses. The limitations of the Sobolev cubatures start as the dimension $m\gg1$ of the domain becomes too high and the complexity in Theorem~\ref{theo:complex}, $i)$ is dominated by the $\Oc(s^2)$, $s =(n+1)^m$ term. We continue the comparison empirically in the next section.

\section{Numerical Experiments}\label{sec:NUM}
We designed the following experiments for validating and demonstrating our theoretical findings. All experiments were
executed on the NVIDIA V100 cluster at HZDR.
Precomputation of the Sobolev cubature matrices is realized in \cite{repo}
as a feature of the open source package  \cite{minterpy}, which is based on our recent work
\cite{PIP1,PIP2,MIP,IEEE}.
For comparison we
benchmark the following schemes:
\begin{enumerate}
    \item[i)]  \emph{Classic PINNs (PINN)} as proposed in \cite{RAISSI2019686} resting on the
     strong $L^2$-MSE loss, Eq.~\eqref{L2}.
    \item[ii)]  \emph{Inverse Dirichlet Balancing (ID-PINNs)} with the $L^2$-MSE loss \cite{maddu2021}, as described in the introduction.
    \item[iii)] \emph{Sobolev Cubature PINNs (SC-PINNs)} with the weak $L^2$-loss for all the experiments unless specified otherwise.
    \item[iv)] \emph{Variational PINNs (VPINNs)} with the strong $L^2$-loss with and without domain decomposition, as introduced in Section\ref{sec:VPINN}. While VPINNs resulted in  incommensurate training effort in 2D only 1D experiments were executed.
\end{enumerate}
For a given ground truth function $g : \Omega \lo \R$ and a $\hat u_\theta$ approximation we measure the $l_1,l_\infty$-errors $\epsilon_1 := \|\mathfrak{g}-\mathfrak{u}\|_1$, $\epsilon_\infty:= \|\mathfrak{g}-\mathfrak{u}\|_\infty$ by sampling on
equidistant grids of size $N =100^2$ in 2D.
For (inverse) parameter inference problems we denote the parameter error  with  $\epsilon_\lambda =|\lambda -\lambda_{gt}|$.

All models are trained with the same number of training points regardless of their specification (random or Legendre grids).

% We denote by $\epsilon_{\infty}$ the $L^\infty-$error of the prediction and by $\epsilon_{l^2}$ the MSE error of the prediction defined as follows. For $u$ the analytic solution and $X_v$ a discrete set of random validation points of size $N$:
% \begin{equation}
%     \begin{cases}
%     \epsilon_{\infty}:=\textrm{max}_{x\in X_v}|\hat{u}(x)-u(x)|\\
%     \epsilon_1:=\frac{1}{N}\sum_{i=0}^{N}(\hat{u}(x_i)-u(x_i))^2
%     \end{cases}
% \end{equation}
\subsection{Poisson Equation}\label{sec:P}
We start by addressing the Poisson problem in dimension $m=2$
\begin{equation*}
   \li\{ \begin{array}{rll}
       -\Delta u(x) -f(x) &= 0  &,  \forall x\in\Omega=[-1,1]^2  \\
         u(x)  -g(x)      &= 0  &,  \forall   x\in\partial\Omega\,,
\end{array}\re.
\end{equation*}
described in detail Eq.~\eqref{eq:PP}.
We choose $f(x,y):= -2\lambda^2\textrm{cos}(\lambda x)\textrm{sin}(\lambda y)$ and $g(x):= \textrm{cos}(\lambda x)\textrm{sin}(\lambda y)$, with frequency $\lambda = 2\pi q$, $q=6$, yielding $u(x,y) =g(x,y)$ to be the analytic solution.

\begin{figure}[t!]
    \centering
         \begin{subfigure}
         \centering
         \includegraphics[width=0.47\textwidth]{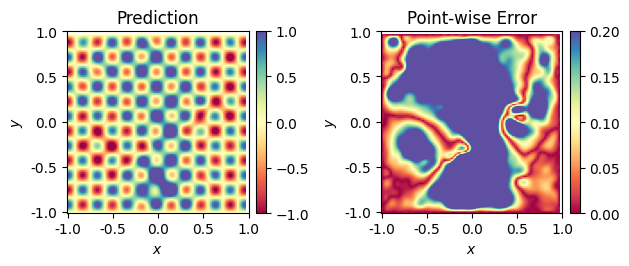}
         \caption{ID-PINN with MSE loss, reaching
         $\epsilon_{\infty} = \mbox{1.24}$, $\epsilon_1 = \mbox{2.22e-1}$, $t\approx  1173$. \label{fig:P_1}}
     \end{subfigure}
     \begin{subfigure}
         \centering
         \includegraphics[width=0.47\textwidth]{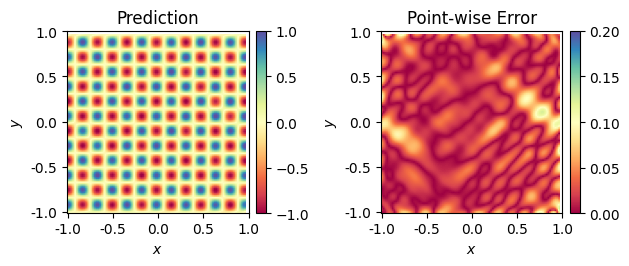}
    \caption{SC-PINN with strong Sobolev loss and A.D., reaching
    $\epsilon_{\infty} = \mbox{1.37e-1}$, $\epsilon_1 =\mbox{2.27e-2}$, $t\approx 725$.    \label{fig:P_2}}
     \end{subfigure}
     \begin{subfigure}
         \centering
         \includegraphics[width=0.47\textwidth]{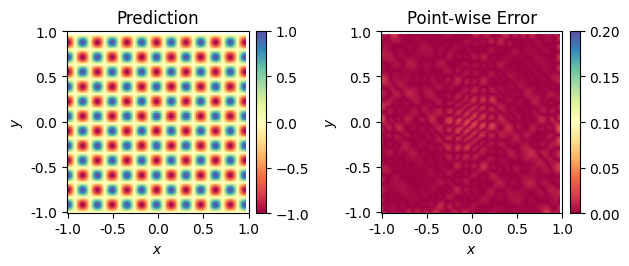}
    \caption{SC-PINN with strong Sobolev loss and P.D., reaching
    $\epsilon_{\infty} = \mbox{3.00e-2}$, $\epsilon_1 = \mbox{3.75e-3}$,  $t\approx 192$.  \label{fig:P_3}}
     \end{subfigure}
     \hfill
     \begin{subfigure}
         \centering
    \includegraphics[width=0.47\textwidth]{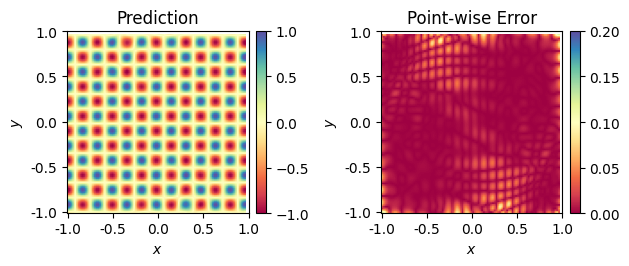}
    \caption{SC-PINN with strong variational Sobolev loss and P.D., reaching
    $\epsilon_{\infty} = \mbox{1.25e-1}$, $\epsilon_1 = \mbox{7.70e-3}$, $t\approx 188$. \label{fig:P_4}}
     \end{subfigure}
\end{figure}
We compare the performance of the following PINN implementations:
\begin{enumerate}
\item[I)] PINN and ID-PINN relying in strong \emph{mean square} (MSE) loss $\Lc_{\mathrm{MSE}}$, Eq.~\eqref{L2}.
\item[II)] SC-PINN with strong Sobolev loss $\Lc_{\mathrm{strong}}$, Eq.~\eqref{eq:strong},
$$\Lc_{\mathrm{strong}}(u) = r_{\mathrm{strong},n_r,k}(u) + s_{\mathrm{strong},n_s,l}(u)$$
for $k=0$, $l=2$,  $n_r = 30$, $n_s = 100$, in two versions: Once computing
$\Delta u$ with automatic differentiation (A.D.) and once without A.D.,
but with \emph{polynomial differentiation} (P.D.) given in Eq.~\eqref{eq:R}.
\item[III)] SC-PINN with strong variational Sobolev loss, Eq.~\eqref{eq:strongvar} $\Lc_{\mathrm{strong}}^{\mathrm{var}}$
$$\Lc_{\mathrm{strong}}^{\mathrm{var}}(u) = r_{\mathrm{strong},n_r,k}^{\mathrm{var}}(u) + s_{\mathrm{strong},n_s,l}^{\mathrm{var}}(u)$$
for $k=0$, $l=0$, $n_r = 30$, $n_s = 100$, without A.D., but with  P.D..
% \item[IV)] SC-PINN with \red{weak variational loss}.
\end{enumerate}
All methods were implemented for fully connected feed-forward NNs, $\hat u \in \Xi_{2,1}$, of $5$ hidden layers, each of $50$ units length, unless specified otherwise. Activation functions were chosen as $\sigma(x) = \sin(x)$, which performed best compared to trials with ReLU, ELU or $\sigma(x)=\tanh(x)$.
All PINNs were trained by applying the Adam optimizer \cite{kingma2014} for $\mbox{30000}$ iterations, batch size $\mathop{bs}= |P_{m,n}|$ for ID-PINN equals SC-PINN, and learning rate of $\mathop{lr}=1\textrm{e}-3$,
whereas ID-PINN applies its dynamic gradient balancing scheme.
% and SC-PINNs gradients were synchronized based on Theorem~\ref{scaling_theorem} and Eq.~\eqref{eq:Rstar}.

Approximation errors and CPU-training times $t$ are reported in   Figs.~\ref{fig:P_1}--\ref{fig:P_4}. While the classic PINN approach failed to converge (reach reasonable approximations) its results are skipped.

In comparison SC-PINN reaches several orders of magnitude better approximations: SC-PINN with $\Lc_{\mathrm{strong}}$ and A.D. improves by one order, while replacing A.D. with P.D. even increases the accuracy by one further order. In addition the CPU runtime is reduced by three fold when executing SC-PINN with P.D. instead of A.D. The choice of  $\Lc_{\mathrm{strong}}^{\mathrm{var}}$ improves the $\epsilon_1$ error by one order of magnitude compared to
SC-PINN with $\Lc_{\mathrm{strong}}$ and A.D., which requires more CPU time.

We address a further prominent PDE problems to continue our empirical investigations. From here on we only use P.D. instead of A.D. for executing the experiments with the SC-PINNs.

% Comparing the MSE with I.D. in figure \ref{fig:poisson_mse} and the Strong Sobolev, figure \ref{poisson_ad}, with A.D. we can observe a significant improvement in the solution, with one order of magnitude reduction on the $\epsilon_{\infty}$ and the $\epsilon_1$ errors.There is even a further improvement comparing the MSE with I.D. against the Strong Sobolev without A.D. figure \ref{poisson_wad}, as the $\epsilon_{\infty}$ and the $\epsilon_1$ error decreases by two orders of magnitude and the computational time is reduced by half.\\
% In this particular example the variational formulation, figure \ref{poissom_var}, performs similarly on the $\epsilon_1$ error compared to the strong Sobolev loss, while having a bigger $\epsilon_{\infty}$ error.

\subsection{Quantum Harmonic Oscillator}

The time independent \emph{Quantum Harmonic Oscillator} in dimension $m=2$, corresponds to the \emph{Schrödinger equation} with the linear potential $V(u(x)):=(x_1^2 + x_2^2)u(x)$, $u\in C^2(\Omega,\mathbb{R})$, see e.g. \cite{richard1980,griffiths2018},
given by:
\begin{equation*}
   \li\{\begin{array}{rll}
       -\Delta u(x) + V(u(x))  &= \lambda u(x)  &,  \forall x\in\Omega  \\
         u(x)  -g(x)     &= 0   &,  \forall x\in\partial\Omega\,,
\end{array}\re.
\end{equation*}

We consider the eigenvalue problem $\lambda = n_1 + n_2 + 1$, $n_1,n_2 \in \N$ with eigenfunctions
\begin{equation*}
g(x_1,x_2) = \frac{\pi^{-1/4}}{\sqrt{2^{n_1+n_2}n_1!n_2!}}e^{-\frac{(x_1^2+x_2^2)}{2}}H_{n_1}(x_1)H_{n_2}(x_2),
\end{equation*}
whereas $H_n$ denotes the $n$-th \emph{Hermite polynomial}.

We keep the experimental design from Section~\ref{sec:P}, choose $\lambda = 15$ and report the results in Fig.~\ref{fig:S1}--\ref{fig:S3}:
\begin{figure}[t!]
    \centering
     \begin{subfigure}
         \centering
         \includegraphics[width=0.47\textwidth]{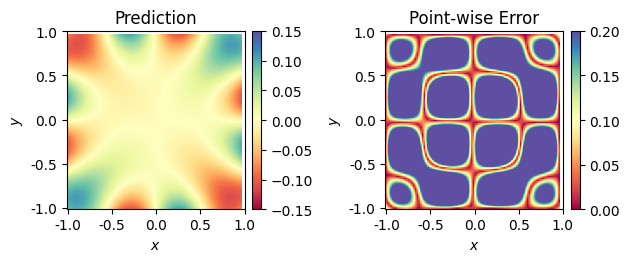}
    \caption{ID-PINN with MSE loss and A.D., reaching $\epsilon_{\infty} = \mbox{1.46e-1}$, $\epsilon_1 = \mbox{4.78e-2}$, $t\approx 905$,  \label{fig:S1}}
     \end{subfigure}
     \begin{subfigure}
         \centering
         \includegraphics[width=0.47\textwidth]{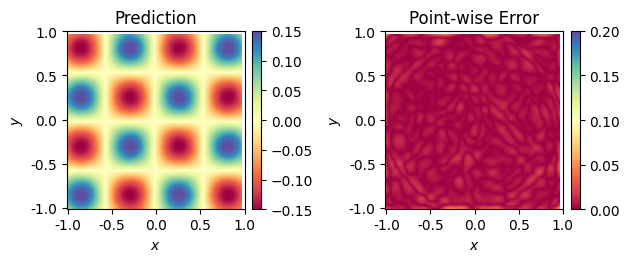}
    \caption{SC-PINN with strong Sobolev loss and P.D., reaching, $\epsilon_{\infty} = \mbox{1.22e-2}$, $\epsilon_1=\mbox{1.24e-3}$, $t\approx 165$,  \label{fig:S2}}
     \end{subfigure}
     \begin{subfigure}
         \centering
     \includegraphics[width=0.47\textwidth]{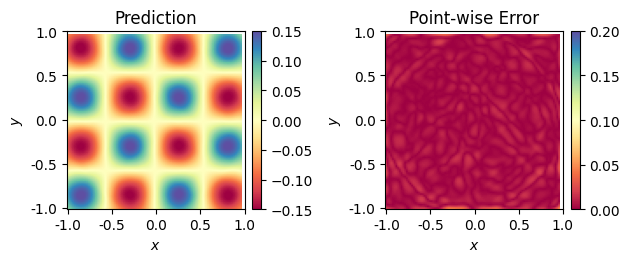}
    \caption{SC-PINN with strong variational Sobolev loss and P.D., reaching, $\epsilon_{\infty} = \mbox{7.27e-3}, \epsilon_1= \mbox{8.16e-4}, t\approx167$, \label{fig:S3}}
     \end{subfigure}
\end{figure}
Similar to the previous experiment classic PINN failed to converge, SC-PINNs improve the accuracy up to 2 orders of magnitude in 4 fold less runtime, whereas, SC-PINN with $\Lc_{\mathrm{strong}}^{\mathrm{var}}$ and P.D. performs best.

% slightly improves compared to
% SC-PINN with $\Lc_{\mathrm{strong}}$ and A.D., but requires most GPU time.
% While the MSE with Inverse Dirichlet, figure \ref{schr_MSE}, fails completely to recover the solution, both the strong Sobolev loss, figure \ref{schr_strong}, and the strong variational loss, figure \ref{schr_strong}, recover the solution accurately. In this example the strong variational loss performs one order of magnitude better that the strong Sobolev one.

\subsection{Poisson problems with hard transitions}

We re-investigate PINN-solutions of the Poisson problem in dimension $m=1$, whose analytic solutions include hard transitions. That is, choosing
$$f(x):= C\big(A\omega^2 \sin(\omega x)+2\beta^2\mathrm{sech}^2(x)\textrm{tanh}(\beta x)\big)\,,$$
with boundary condition $g(x):= C(A\sin(\omega x)+\tanh(\beta x))$
yielding $u(x) =g(x)$ to be the analytic solution. Two scenarios were considered:
\begin{align*}
  S_1 &=\{C=0.1,A=0.1,\beta = 30,\omega= 20\pi, \mathrm{bs}=100\}\\
  S_2 &=\{C=0.1,A=0.1,\beta = 5,\omega= 26.5\pi, \mathrm{bs}=100\}
%   \red{S_3} &\red{=\{C=0.1,A=0.1,\beta = 30,\omega= 20\pi, \mathrm{bs}=100\}}
\end{align*}
\begin{figure}[t!]
    \centering
     \begin{subfigure}
         \centering
         \includegraphics[width=0.235\textwidth]{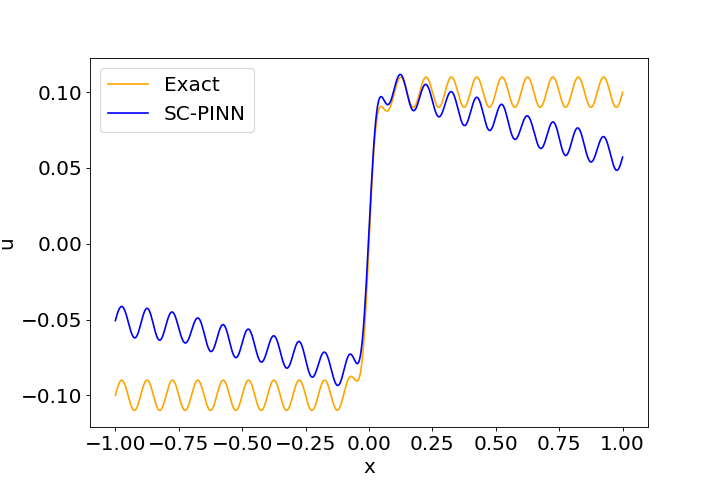}
         \includegraphics[width=0.235\textwidth]{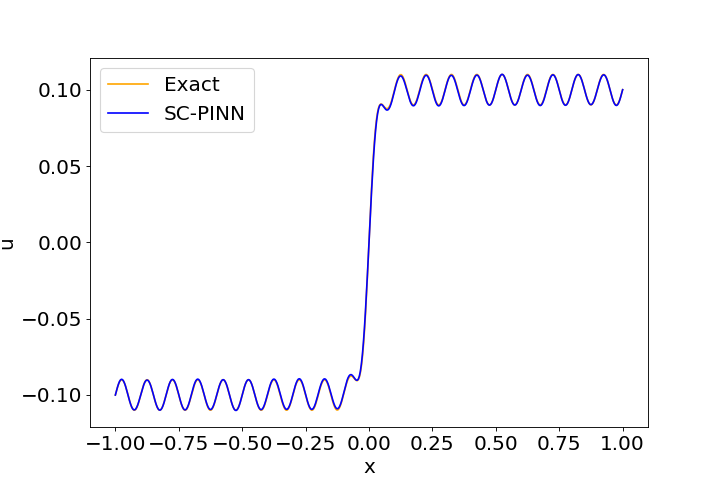}
    \caption{SC-PINN with strong Sobolev loss and P.D. (left), reaching   $\epsilon_{\infty} =\mbox{3.04}e-2$, $\epsilon_1 =\mbox{7.24}e-2,t\approx 150$,
    SC-PINN with strong variational Sobolev loss and P.D.(right), reaching   $\epsilon_{\infty} =\mbox{2.0e-3}$, $\epsilon_1 =\mbox{4.0e-4}, t\approx 151$,
    scenario $S_1$.
    \label{1d_P1}}
     \end{subfigure}
    %  \begin{subfigure}
    %      \centering
    %      \includegraphics[width=0.2\textwidth]{poisson1d_ManualGD_var_final.png}
    % \caption{Strong Variational loss, $\epsilon_{\infty} =0.002,\epsilon_1=0.0004$, Parameters: $S_1$   \label{1d_poisson_var}}
    %  \end{subfigure}
     \begin{subfigure}
         \centering
         \includegraphics[width=0.23\textwidth]{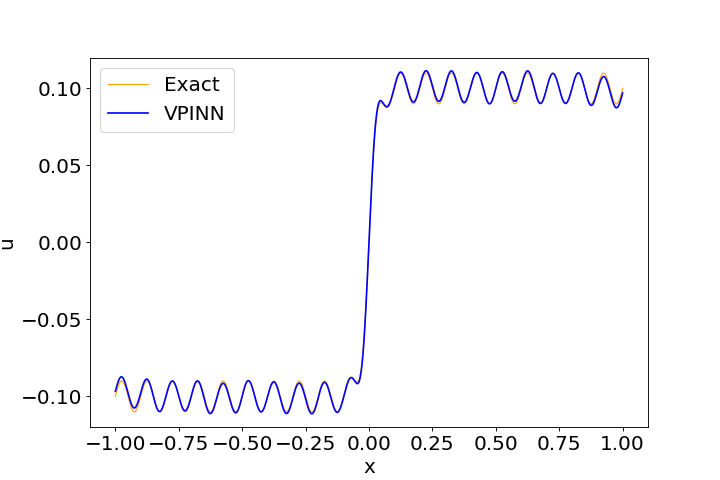}
         \includegraphics[width=0.235\textwidth]{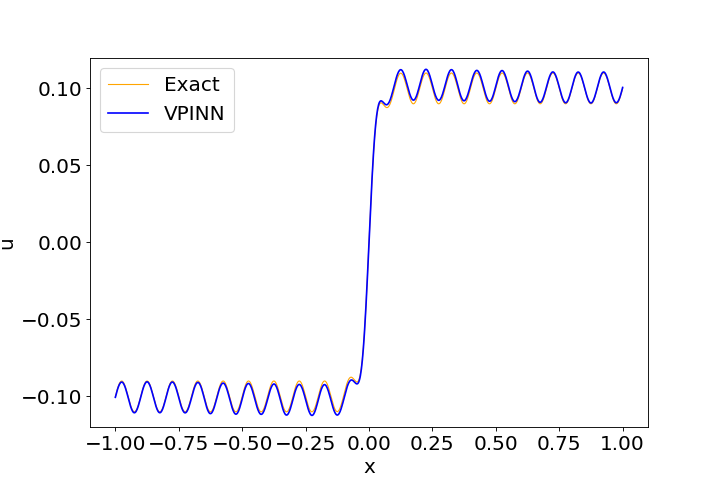}
    \caption{VPINN with $N_\mathrm{el}=1$ (left), reaching  $\epsilon_{\infty} = \mbox{3.29e-3}$, $\epsilon_1=\mbox{9.94e-4}, t\approx 96$,
    VPINN with $N_\mathrm{el}=3$ (right), reaching
    $\epsilon_{\infty} = \mbox{2.73e-3}$, $\epsilon_1=\mbox{1.40e-3}, t\approx 191$, scenario $S_1$.     \label{1d_P2}}
     \end{subfigure}
    %  \begin{subfigure}
    %      \centering
    % \includegraphics[width=5cm]{VPINN_S1_final_3(3).png}
    % \caption{VPINN with $N_\mathrm{el}=3$ and strong variational loss,  $\epsilon_{\infty} =0.0019,\epsilon_1=0.00043$, Parameters: $S_1$}
    % \label{1d_poisson_VPINN2}
    %  \end{subfigure}
    %  \hfill
     \begin{subfigure}
         \centering
         \includegraphics[width=0.235\textwidth]{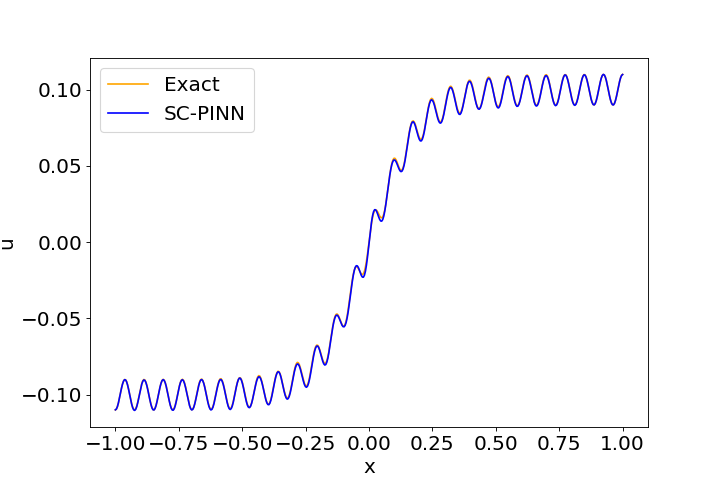}
         \includegraphics[width=0.235\textwidth]{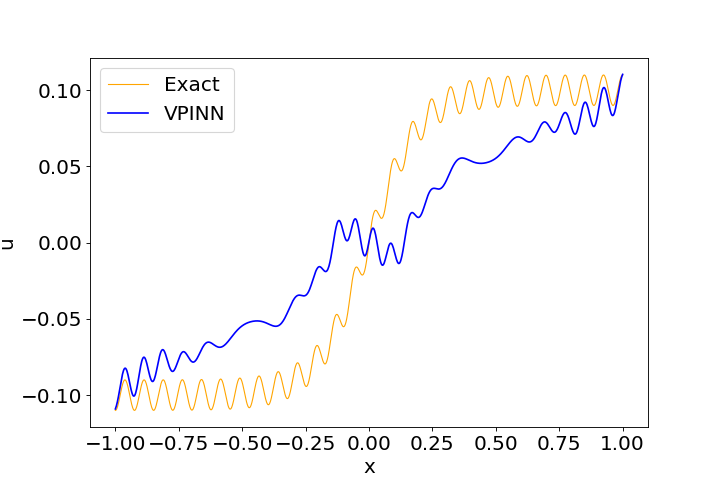}
        %  VPINN_26.5_3E_final_3.png}
    \caption{SC-PINN with strong variational Sobolev loss and P.D.(left), reaching   $\epsilon_{\infty} =\mbox{2.75e-3}$, $\epsilon_1 =\mbox{5.35e-4},t\approx 151$,
    % Strong Variational loss, $\epsilon_{\infty} =0.0027,\epsilon_1 =0.0005$,
     VPINN with $N_\mathrm{el}=3$, reaching
    $\epsilon_{\infty} = \mbox{6.50e-2}$, $\epsilon_1=\mbox{3.40e-2}$,
    scenario $S_2$, $t\approx 180$.\\}
    \label{1d:PS2}
     \end{subfigure}
    %  \begin{subfigure}
    %      \centering
    %      \includegraphics[width=5cm]{VPINN_26.5_3E_final.png}
    % \caption{VPINN with $N_\mathrm{el}=1$ and strong variational loss, $\epsilon_{\infty} =0.012,\epsilon_1 =0.0063$, Parameters: $S_3$}
    % \label{poisson1d_VPINN}
    %  \end{subfigure}
    %  \begin{subfigure}
    %      \centering
    %      \includegraphics[width=5cm]{VPINN_26.5_3E_final_3.png}
    % \caption{VPINN with $N_\mathrm{el}=3$ and strong variational loss, $\epsilon_{\infty} =0.0075,\epsilon_1 =0.0036$, Parameters: $S_3$}
    % \label{poisson1d_VPINN3}
    %  \end{subfigure}
     \begin{subfigure}
         \centering
    \includegraphics[width=0.236\textwidth]{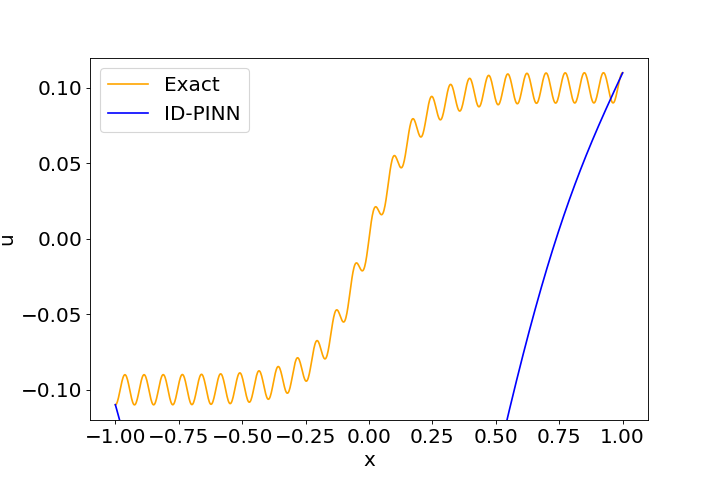}
     \includegraphics[width=0.236\textwidth]{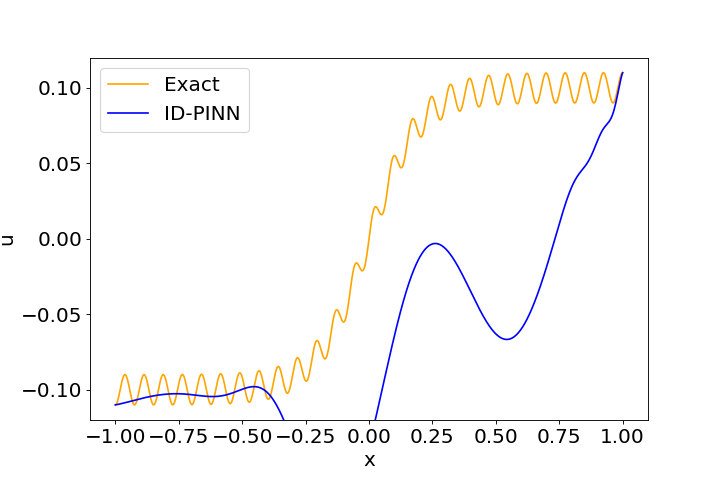}
    \caption{ID-PINN with MSE loss, $\mathrm{bs}=500$ and A.D.,(left), reaching $\epsilon_{\infty}=\mbox{4.74e-1}$, $\epsilon_1= \mbox{2.50e-1}, t\approx 158$,
    ID-PINN with MSE loss without I.D., $\mathrm{bs}=5000$ and A.D.,(right), reaching $\epsilon_{\infty}=\mbox{1.73e-1}$, $\epsilon_1= \mbox{7.13e-2}, t\approx 360$, scenario $S_2$. \label{poisson1d_MSE1}}
     \end{subfigure}
%      \begin{subfigure}
%          \centering
%   \includegraphics[width=5cm]{MSE_Poisson_1dMSE_5000_final.png}
%     \caption{MSE loss without I.D., $\epsilon_{\infty}=0.025, \epsilon_1=0.015$, $N_\textrm{points}=5000$}
%     \label{poisson1d_MSE2}
%      \end{subfigure}
%      \hfill
\vspace{-0.5cm}
\end{figure}

Up to reducing to $4$ hidden layers, each of $20$ units length the same NN architecture, $\hat u \in \Xi_{1,1}$,
as prior was chosen. The degree $n \in \N$ of the Sobolev losses  is set as the batch size $n=\mathrm{bs}$.
In addition to the previous PINN methods we consider the
VPINNs,
% \cite{kharazmi2019variational}, \cite{Kharazmi2020hpVPINNsVP},
introduced in Section~\ref{sec:VPINN}, relying on the strong variational Sobolev loss and domain decomposition specified by the number of its elements $N_{\mathrm{el}} \in \N$.

Comparison of SC-PINNs and VPINNs for scenario $S_1$ is reported in Figs.~\ref{1d_P1},\ref{1d_P2}: While SC-PINN with strong Sobolev loss
misses to capture the solution, SC-PINN with variational loss reaches compatible approximations compared to VPINN with $3$-element decomposition $N_\mathrm{el}=3$.

% We can observe first how in this case the strong Sobolev loss figure \ref{1d_poisson_sob}, fails to recover the solution accurately, specially at the boundary, while the strong variational formulation figure \ref{1d_poisson_var} reduces the error by two orders of magnitude. With respect to the VPINN, we can see how the approach without domain decomposition  figure \ref{1d_poisson_VPINN1} has a high error compare to the strong variational loss. The VPINN with 3 elements for domain decomposition  figure \ref{1d_poisson_VPINN2} show an improvement with respect to the VPINN without domain decomposition, but still the error is higher that the strong variational loss.\\
% We present a second test with the following set of parameters, $S_2:=\{C=0.1,A=0.1,\beta = 30,\omega= 26.5\pi, N_\textrm{points}=100\}$. Finally, we observed that the MSE loss with I.D. becomes unstable, so we include the MSE without balancing with 5000 points, to show the improvement of replacing the MSE as a loss for the Sobolev cubatures.

For the second set of parameters $S_2$ results are reported in Fig.~\ref{1d:PS2}:
We observe a similar behaviour as in scenario $S_1$, but SC-PINN with strong variational loss reaches one order of magnitude better (overall) $\epsilon_1$-error compared to VPINN.
ID-PINN and PINN both fail to converge. We present the results only for the I.D. balancing, once for batch size $\mathrm{bs}=500$ and once for $\mathrm{bs}=5000$,
are reported in Fig.~\ref{poisson1d_MSE1}. We, however, observe that even by increasing the batch size by a factor of $10$ ID-PINN does not become compatible to SC-PINN.

% The importance of gradient balancing is finally pointed out next.
% crucial for any PINN approach \cite{wang2021understanding} we finally focus on this challenge.

\subsection{2D Poisson Inverse Problem}
We will now consider the inverse problem for the 2D Poisson equation, where we want to infer $\lambda$, that appears explicitly on the RHS $f(x) = \lambda \cos(\omega x)\sin(\omega y)$ for $\omega = \pi$. We used for the SC-PINN, degree $100$ and $30$ quadrature for the boundary and the domain respectively and the same amount of points randomly sampled for the ID-PINN and the standard PINN. The ground truth $\lambda_{gt} = 2\omega^2$ and the weak $L^2-$loss. We test it against the standard PINN and the ID-PINN with same number of points.
\newcommand{\tabIPP}{% Just for this example
\begin{tabular}[t]{llll}
Method & \multicolumn{2}{c}{Approximation error}& Runtime (s)\\
& $\epsilon_{\lambda}$ & $\epsilon_{1}$\\
\hline
  PINN    & $4.63\cdot10^{-1}$ & $1.13\cdot 10^{-2}$ & $t\approx 1592$\\
  ID-PINN & $2.14\cdot10^{-2}$ &$8.09\cdot10^{-4}$ & $t\approx 2184$ \\
  SC-PINN  & $\mathbf{3\cdot10^{-4}}$ & $\mathbf{5.49\cdot10^{-4}}$ & $t \approx \mathbf{103}$ \\
& & &\\
& & & \\
& & & \\
& & &
\end{tabular}
}

\begin{table}[ht]
\begin{minipage}[b]{0.5\textwidth}%
\centering
\includegraphics[width=0.7\textwidth]{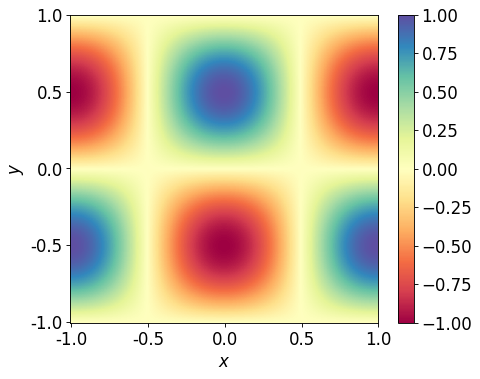}
\captionof{figure}{Solution for 2D Poisson with $\lambda_{gt} = \pi$.}%
\end{minipage}
\begin{minipage}[b]{.5\textwidth }%
\tabIPP
\vspace{-45pt}
\caption{Errors for 2D Poisson inverse problem}\label{IP_QHO}%
\end{minipage}%
\end{table}
Due to Table \ref{IP_QHO}, the SC-PINN recovers the parameter $\lambda$ with one order of magnitude higher accuracy  by requiring almost two orders of magnitude less runtime. even in the inverse problem setting, this shows the superiority of SC-PINN in both approximation and computational performance. To support this result, we consider the inverse problem for the time independent 2D QHO.

\subsection{2D QHO Inverse Problem}
We pose the QHO eigenvalue problem for unknown eigenvalue $\lambda_{\mathrm{gt}}=5$ and seek finding the eigenvalue and the PDE solution simultaneously. In this setting we use the SC-PINN with $L^2-$weak loss, and a $200$ degree cubature for the boundary and use degree $50$ in the domain. We  compare it with the ID-PINN and the standard PINN trained with the same number of (random) training points.
\newcommand{\tabQI}{% Just for this example
\begin{tabular}[t]{llll}
Method & \multicolumn{2}{c}{Approximation error}& Runtime (s)\\
& $\epsilon_{\lambda}$ & $\epsilon_{1}$\\
\hline
  PINN    & $6.01$ & $7.32\cdot 10^{-2}$ & $t\approx 1414$\\
  ID-PINN & $6.21\cdot10^{-2}$ &$7.51\cdot10^{-3}$ & $t\approx 1346$ \\
  SC-PINN  & $\mathbf{2.18\cdot10^{-4}}$ & $\mathbf{5.68\cdot10^{-4}}$ & $t \approx \mathbf{192}$ \\
& & &\\
& & & \\
& & &
\end{tabular}
}

\begin{table}[ht]
\begin{minipage}[b]{.5\textwidth}%
\centering
\vspace{-5pt}\includegraphics[width=0.7\textwidth]{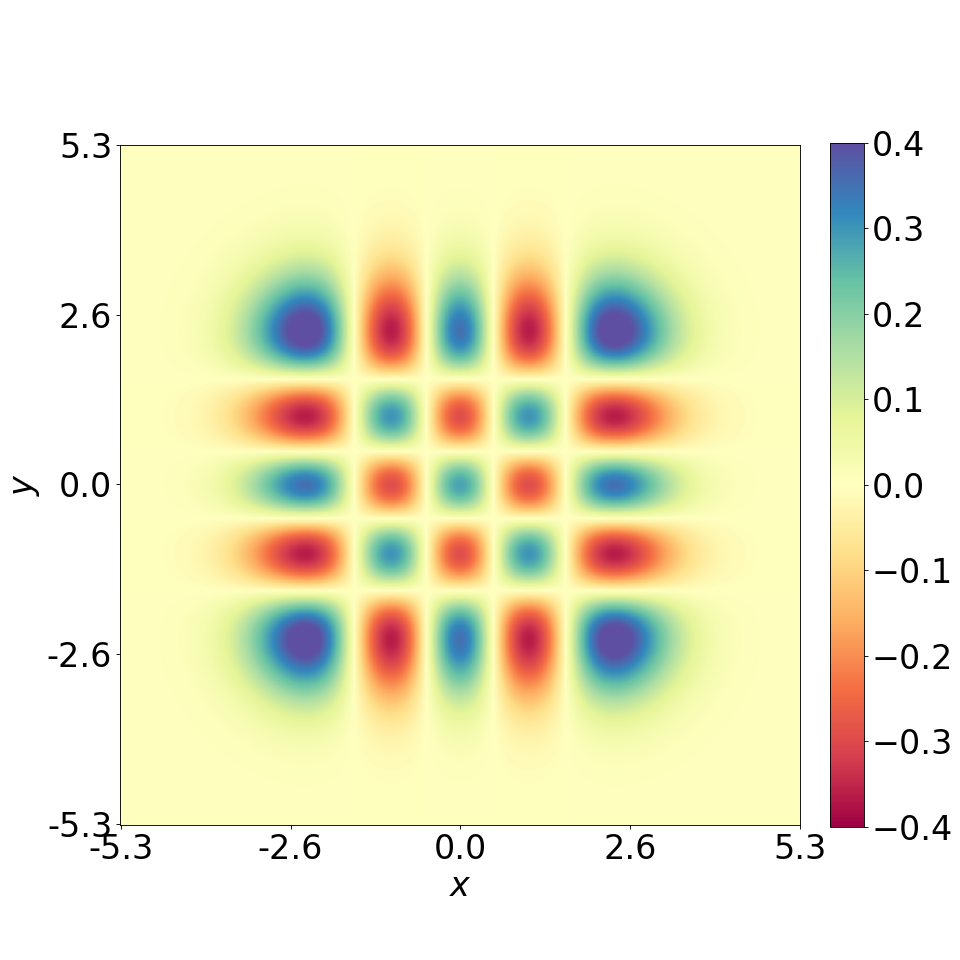}
\captionof{figure}{Solution for 2D QHO with $\lambda_{gt} = 5$.}%
\end{minipage}
\begin{minipage}[b]{.50\textwidth }%
\tabQI
\vspace{-30pt}
\caption{Errors for 2D QHO inverse problem with $\lambda_{gt}=5$}\label{IP_QHO_I}%
\end{minipage}%
\end{table}
Table \ref{IP_QHO_I} shows that SC-PINN outperforms the ID-PINN and the PINN by inferring $\lambda$ with four orders of magnitude better accuracy and recovers the PDE solution with one order of magnitude higher precision. On top, SC-PINN requires two orders of magnitude less runtime. This demonstrates once more the superiority of the method SC-PINN when applied to forward and inverse problems.

\subsection{Non-linear Burger's Equation in 1D}
We consider the time independent Burger's equation in conservative form with Dirichlet boundary conditions given by
    \begin{equation*}
   \li\{\begin{array}{rll}
       -\frac{d^2}{dxdx} u(x) +  \frac{1}{2}\frac{d}{dx} (u(x)^2)  &= f(x) &,  \forall x\in\Omega\\
         u(x)  -g(x)     &= 0   &,  \forall x\in\partial\Omega\,,
\end{array}\re.
\end{equation*}
with $f(x) := \frac{\omega}{2}\textrm{sin}(2\omega x)+\omega^2\textrm{sin}(\omega x)$\,.

We solve the PDE with a 100 degree quadrature, $\omega = 14\pi$ and the and strong variational loss for the SC-PINN norm. We test it against the PINN and the ID-PINN with same number of (random) training points.
\newcommand{\tabBF}{% Just for this example
\begin{tabular}[t]{llll}
Method & \multicolumn{2}{c}{Approximation error}& Runtime (s)\\
& $\epsilon_{1}$ & $\epsilon_{\infty}$\\
\hline
  PINN    & $2.2\cdot10^{-1}$ &$4.2\cdot10^{-1}$ & $t\approx 125$\\
  ID-PINN & $7.9\cdot10^{-1}$ &$2.6$ & $t\approx 138$ \\
  SC-PINN & $2.7\cdot10^{-3}$ &$8.7\cdot10^{-3}$ & $t\approx 157$\\
& & &\\
& & &
\end{tabular}
}

\begin{table}[ht]
\centering
\begin{minipage}[b]{.5\textwidth}%
\centering
\vspace{10pt}
\includegraphics[width=0.8\textwidth]{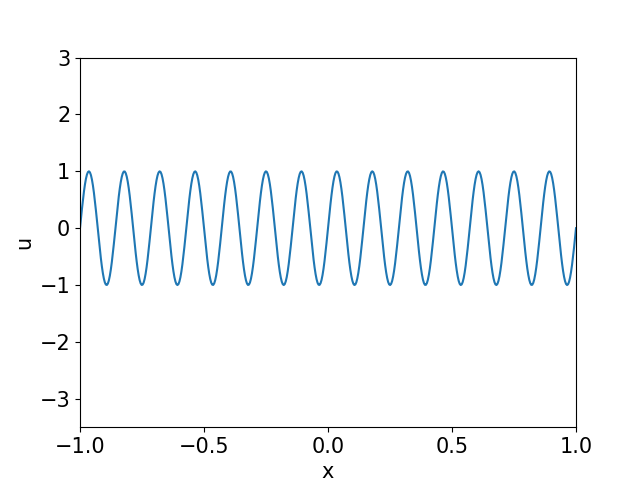}
\captionof{figure}{Solution for 1D Burger's with $\omega = 14\pi$.\label{x}}%
\end{minipage}
\begin{minipage}[b]{.50\textwidth }%
\tabBF
\vspace{-30pt}
\caption{Errors for 1D Burger's forward problem}\label{FW_BRG}%
\end{minipage}%
\end{table}
According to the results presented in Table \ref{FW_BRG} SC-PINN performs again better than the other PINN formulations by reaching 2 orders of magnitude smaller $\epsilon_1$-error. However, SC-PINN is slightly slower than the other PINNs in this 1D-setting. We focus on the runtime performance in a separate experiment below.

\newcommand{\tabPAD}{% Just for this example
\begin{tabular}[t]{llll}
Method & \multicolumn{2}{c}{Approximation error}& Runtime (s)\\
& $\epsilon_{\lambda}$ & $\epsilon_{1}$\\
\hline
  PINN    & $2.52\cdot10^{-1}$ & $3.93\cdot 10^{-2}$ & $t\approx 278.7$\\
  SC-PINN  & $\mathbf{1.4\cdot10^{-2}}$ & $\mathbf{2.8\cdot10^{-2}}$ & $t \approx \mathbf{25.72 }$ \\
& & &\\
& & & \\
& & &
\end{tabular}
}

\begin{table}[t!]
\begin{minipage}[b]{.4\textwidth}%
\centering
\hspace*{20pt}\includegraphics[width=0.9\textwidth]{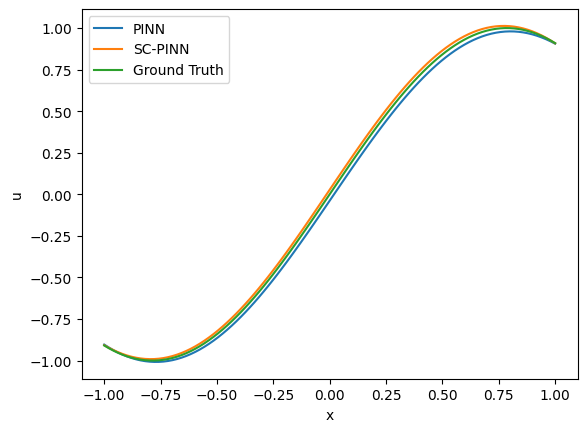}
\captionof{figure}{Solution of 4-th order ODE}%
\end{minipage}
\begin{minipage}[b]{.50\textwidth}%
\tabPAD
\vspace{-30pt}\caption{Errors for the 4th order ODE}\label{PAD}%
\end{minipage}%
\end{table}

\begin{subsection}{Polynomial Differentiation vs Automatic Differentiation}\label{sec:exAD}
We consider the forward problem of the following $4$th order ODE:
\begin{equation*}
    \frac{d^4}{dx}y = f(x), \forall x\in\Omega\,,
\end{equation*}
where $f(x) := -\omega^4 \sin(\omega x)$ with $y(0) = \sin(\omega 0)$. We benchmark the  SC-PINN with P.D. against the standard PINN with MSE loss, $bs=101$ training points, $10000$ epochs executed by the Adam optimizer \cite{kingma2014}.

\end{subsection}
As reported in Table~\ref{PAD} the SC-PINN recovers a more accurate solution of the 4th order ODE in one order of magnitude less runtime. This suggests that replacing A.D. by P.D. increases the efficiency, as predicted in our discussions in Section~\ref{sec:ADPD}, Example~\ref{ex:AD}.

\section{Conclusion}
We introduced the notion of Sobolev cubatures and gave theoretical insights in order to setup the novel Sobolev PINNs (SC-PINNs). As predicted, by the runtime complexity analysis we did, for low dimensional problems this results
in a faster PDE learning scheme than PINNs relying on automatic differentiation.

Moreover, in Theorem~\ref{theo:ss} we theoretically ensured that the SC-PINNs
converge to strong (smooth) PDE solutions for well posed problems. This result is complemented by the several order of magnitude higher accuracy the SC-PINNs reached when considering prominent  linear, non-linear, forward, and inverse PDE problems.

% nin higher approximation power in practice.

% % presented herein  framework improves, by one or two orders of magnitude, the accuracy of the solution with respect to methods such as the VPINNs or the MSE loss with and without loss balancing.
% This is due to several contributions that come with the framework: We replaced automatic differentiation with polynomial differentiation, provided uniform convergence guarantees, enabled variational losses for weak PDE formulations and present an analytic control the synchronize gradient unbalances of the total loss.

Depending on the numerical experiment, the choice of the Sobolev cubature differed. While we  meanwhile deepened the theoretical insights presented in this article to deliver the optimal choice  beforehand these subjects are out of scope, here, and part of a follow-up study. This includes a relaxation of the Sobolev cubatures, resisting the curse of dimensionality when addressing higher dimensional problems.

Apart from these potential enhancements the class of low dimensional, $\dim=2,3,4$, real world problems is huge. Thus, the demonstrations, here, make us believe that applying the SC-PINNs might be beneficial for many scientific applications across all disciplines.
\section{Acknowledgements}
We would like to thank Dominik Sturm from the Center for Advanced Systems Understanding, for the useful discussions and support during the whole project, as well as for providing the base framework for training the PINNs and the ID-PINNs.
\bibliography{PINNS_REF}

\begin{thebibliography}{37}
\providecommand{\natexlab}[1]{#1}
\providecommand{\url}[1]{\texttt{#1}}
\expandafter\ifx\csname urlstyle\endcsname\relax
  \providecommand{\doi}[1]{doi: #1}\else
  \providecommand{\doi}{doi: \begingroup \urlstyle{rm}\Url}\fi

\bibitem[Adams \& Fournier(2003)Adams and Fournier]{Adams2003}
Adams, R.~A. and Fournier, J.~J.
\newblock \emph{{S}obolev spaces}, volume 140.
\newblock Academic press, 2003.

\bibitem[Anthony \& Bartlett(2009)Anthony and Bartlett]{martin2009}
Anthony, M. and Bartlett, P.~L.
\newblock \emph{Neural network learning: Theoretical foundations}.
\newblock Cambridge University press, 2009.

\bibitem[Arjovsky et~al.(2017)Arjovsky, Chintala, and
  Bottou]{pmlr-v70-arjovsky17a}
Arjovsky, M., Chintala, S., and Bottou, L.
\newblock {W}asserstein generative adversarial networks.
\newblock In Precup, D. and Teh, Y.~W. (eds.), \emph{Proceedings of the 34th
  International Conference on Machine Learning}, volume~70 of \emph{Proceedings
  of Machine Learning Research}, pp.\  214--223. PMLR, 06--11 Aug 2017.
\newblock URL \url{https://proceedings.mlr.press/v70/arjovsky17a.html}.

\bibitem[Baydin et~al.(2018{\natexlab{a}})Baydin, Pearlmutter, Radul, and
  Siskind]{JMLR:v18:17-468}
Baydin, A.~G., Pearlmutter, B.~A., Radul, A.~A., and Siskind, J.~M.
\newblock Automatic differentiation in machine learning: a survey.
\newblock \emph{Journal of Machine Learning Research}, 18\penalty0
  (153):\penalty0 1--43, 2018{\natexlab{a}}.
\newblock URL \url{http://jmlr.org/papers/v18/17-468.html}.

\bibitem[Baydin et~al.(2018{\natexlab{b}})Baydin, Pearlmutter, Radul, and
  Siskind]{baydin2018}
Baydin, A.~G., Pearlmutter, B.~A., Radul, A.~A., and Siskind, J.~M.
\newblock Automatic differentiation in machine learning: a survey.
\newblock \emph{Journal of machine learning research}, 18, 2018{\natexlab{b}}.

\bibitem[Bernardi \& Maday(1997)Bernardi and Maday]{bernardi1997spectral}
Bernardi, C. and Maday, Y.
\newblock Spectral methods.
\newblock \emph{Handbook of numerical analysis}, 5:\penalty0 209--485, 1997.

\bibitem[Brezis(2011)]{brezis2011}
Brezis, H.
\newblock \emph{Functional analysis, Sobolev spaces and partial differential
  equations}, volume~2.
\newblock Springer, 2011.

\bibitem[Canuto et~al.(2007)Canuto, Hussaini, Quarteroni, and
  Zang]{canuto2007spectral}
Canuto, C., Hussaini, M.~Y., Quarteroni, A., and Zang, T.~A.
\newblock \emph{Spectral methods: fundamentals in single domains}.
\newblock Springer Science \& Business Media, 2007.

\bibitem[De~Branges(1959)]{weier2}
De~Branges, L.
\newblock The {Stone-Weierstrass Theorem}.
\newblock \emph{Proceedings of the American Mathematical Society}, 10\penalty0
  (5):\penalty0 822--824, 1959.

\bibitem[Ern \& Guermond(2004)Ern and Guermond]{ern2004theory}
Ern, A. and Guermond, J.-L.
\newblock \emph{Theory and practice of finite elements}, volume 159.
\newblock Springer, 2004.

\bibitem[Goodfellow et~al.(2016)Goodfellow, Bengio, and
  Courville]{goodfellow2016}
Goodfellow, I., Bengio, Y., and Courville, A.
\newblock \emph{Deep learning}.
\newblock MIT press, 2016.

\bibitem[Griffiths \& Schroeter(2018)Griffiths and Schroeter]{griffiths2018}
Griffiths, D.~J. and Schroeter, D.~F.
\newblock \emph{Introduction to quantum mechanics}.
\newblock Cambridge University Press, 2018.

\bibitem[Hecht \& Sbalzarini(2018)Hecht and Sbalzarini]{IEEE}
Hecht, M. and Sbalzarini, I.~F.
\newblock Fast interpolation and {F}ourier transform in high-dimensional
  spaces.
\newblock In Arai, K., Kapoor, S., and Bhatia, R. (eds.), \emph{Intelligent
  Computing. Proc. 2018 IEEE Computing Conf., Vol. 2,}, volume 857 of
  \emph{Advances in Intelligent Systems and Computing}, pp.\  53--75, London,
  UK, 2018. Springer Nature.

\bibitem[Hecht et~al.(2017)Hecht, Cheeseman, Hoffmann, and Sbalzarini]{PIP1}
Hecht, M., Cheeseman, B.~L., Hoffmann, K.~B., and Sbalzarini, I.~F.
\newblock A quadratic-time algorithm for general multivariate polynomial
  interpolation.
\newblock \emph{arXiv preprint arXiv:1710.10846}, 2017.

\bibitem[Hecht et~al.(2018)Hecht, Hoffmann, Cheeseman, and Sbalzarini]{PIP2}
Hecht, M., Hoffmann, K.~B., Cheeseman, B.~L., and Sbalzarini, I.~F.
\newblock Multivariate {N}ewton interpolation.
\newblock \emph{arXiv preprint arXiv:1812.04256}, 2018.

\bibitem[Hecht et~al.(2020)Hecht, Gonciarz, Michelfeit, Sivkin, and
  Sbalzarini]{MIP}
Hecht, M., Gonciarz, K., Michelfeit, J., Sivkin, V., and Sbalzarini, I.~F.
\newblock Multivariate interpolation in unisolvent nodes--lifting the curse of
  dimensionality.
\newblock \emph{arXiv preprint arXiv:2010.10824}, 2020.

\bibitem[Hernandez~Acosta et~al.(2021)Hernandez~Acosta, Krishnan
  Thekke~Veettil, Wicaksono, and Hecht]{minterpy}
Hernandez~Acosta, U., Krishnan Thekke~Veettil, S., Wicaksono, D., and Hecht, M.
\newblock {\sc minterpy} - multivariate interpolation in python.
\newblock \emph{https://github.com/casus/minterpy/}, 2021.

\bibitem[Jost(2002)]{Jost}
Jost, J.
\newblock \emph{Partial Differential Equations}.
\newblock New York: Springer-Verlag, 2002.

\bibitem[Kharazmi et~al.(2019)Kharazmi, Zhang, and
  Karniadakis]{kharazmi2019variational}
Kharazmi, E., Zhang, Z., and Karniadakis, G.~E.
\newblock Variational physics-informed neural networks for solving partial
  differential equations.
\newblock \emph{arXiv preprint arXiv:1912.00873}, 2019.

\bibitem[Kharazmi et~al.(2020)Kharazmi, Zhang, and
  Karniadakis]{Kharazmi2020hpVPINNsVP}
Kharazmi, E., Zhang, Z., and Karniadakis, G.~E.
\newblock hp-vpinns: Variational physics-informed neural networks with domain
  decomposition.
\newblock \emph{ArXiv}, abs/2003.05385, 2020.

\bibitem[Kingma \& Ba(2014)Kingma and Ba]{kingma2014}
Kingma, D.~P. and Ba, J.
\newblock Adam: A method for stochastic optimization.
\newblock \emph{arXiv preprint arXiv:1412.6980}, 2014.

\bibitem[LeVeque(2007)]{LeVeque2007FiniteDM}
LeVeque, R.~J.
\newblock \emph{Finite difference methods for ordinary and partial differential
  equations: steady-state and time-dependent problems}.
\newblock SIAM, 2007.

\bibitem[Li \& Liu(2007)Li and Liu]{li2007meshfree}
Li, S. and Liu, W.~K.
\newblock \emph{Meshfree particle methods}.
\newblock Springer Science \& Business Media, 2007.

\bibitem[Liboff(1980)]{richard1980}
Liboff, R.~L.
\newblock \emph{Introductory Quantum Mechanics}.
\newblock Addison-Wesley Publishing Company. Canad{\'a}, 1980.

\bibitem[Long et~al.(2018)Long, Lu, Ma, and Dong]{Long2018PDENetLP}
Long, Z., Lu, Y., Ma, X., and Dong, B.
\newblock Pde-net: Learning pdes from data.
\newblock \emph{ArXiv}, abs/1710.09668, 2018.

\bibitem[Maddu et~al.(2021)Maddu, Sturm, Müller, and Sbalzarini]{maddu2021}
Maddu, S., Sturm, D., Müller, C.~L., and Sbalzarini, I.~F.
\newblock Inverse dirichlet weighting enables reliable training of physics
  informed neural networks.
\newblock \emph{Machine Learning: Science and Technology}, 2021.
\newblock URL \url{http://iopscience.iop.org/article/10.1088/2632-2153/ac3712}.

\bibitem[Raissi et~al.(2019)Raissi, Perdikaris, and Karniadakis]{RAISSI2019686}
Raissi, M., Perdikaris, P., and Karniadakis, G.
\newblock Physics-informed neural networks: A deep learning framework for
  solving forward and inverse problems involving nonlinear partial differential
  equations.
\newblock \emph{Journal of Computational Physics}, 378:\penalty0 686--707,
  2019.
\newblock ISSN 0021-9991.
\newblock \doi{https://doi.org/10.1016/j.jcp.2018.10.045}.
\newblock URL
  \url{https://www.sciencedirect.com/science/article/pii/S0021999118307125}.

\bibitem[Sirignano \& Spiliopoulos(2018)Sirignano and
  Spiliopoulos]{Sirignano2018DGMAD}
Sirignano, J.~A. and Spiliopoulos, K.
\newblock Dgm: A deep learning algorithm for solving partial differential
  equations.
\newblock \emph{Journal of Computational Physics}, 2018.

\bibitem[Stroud(1971)]{stroud}
Stroud, A.
\newblock \emph{Approximate calculation of multiple integrals: Prentice-Hall
  series in automatic computation}.
\newblock Prentice-Hall (Englewood Cliffs, NJ), 1971.

\bibitem[Stroud(2011)]{stroud2}
Stroud, A.
\newblock Secrest. d.(1966). {G}aussian quadrature formulas, 2011.

\bibitem[Suraz~Cardona(2022)]{repo}
Suraz~Cardona, J.~E.
\newblock Sobolev cubature based {PDE}-learning.
\newblock \emph{https://github.com/casus/{PDE}\_learning/}, 2022.

\bibitem[Trefethen(2017{\natexlab{a}})]{trefethen2017}
Trefethen, L.~N.
\newblock Cubature, approximation, and isotropy in the hypercube.
\newblock \emph{SIAM Review}, 59\penalty0 (3):\penalty0 469--491,
  2017{\natexlab{a}}.

\bibitem[Trefethen(2017{\natexlab{b}})]{trefethen2018}
Trefethen, L.~N.
\newblock Multivariate polynomial approximation in the hypercube.
\newblock \emph{Proceedings of the American Mathematical Society}, 145\penalty0
  (11):\penalty0 4837--4844, 2017{\natexlab{b}}.

\bibitem[Trefethen(2019)]{Trefethen2019}
Trefethen, L.~N.
\newblock \emph{Approximation theory and approximation practice}, volume 164.
\newblock SIAM, 2019.

\bibitem[Wang et~al.(2021)Wang, Teng, and Perdikaris]{wang2021understanding}
Wang, S., Teng, Y., and Perdikaris, P.
\newblock Understanding and mitigating gradient flow pathologies in
  physics-informed neural networks.
\newblock \emph{SIAM Journal on Scientific Computing}, 43\penalty0
  (5):\penalty0 A3055--A3081, 2021.

\bibitem[Weierstrass(1885)]{weier1}
Weierstrass, K.
\newblock {\"U}ber die analytische {D}arstellbarkeit sogenannter
  willk{\"u}rlicher {F}unktionen einer reellen {V}er{\"a}nderlichen.
\newblock \emph{Sitzungsberichte der K{\"o}niglich Preu{\ss}ischen Akademie der
  Wissenschaften zu Berlin}, 2:\penalty0 633--639, 1885.

\bibitem[Yang et~al.(2020)Yang, Zhang, and
  Karniadakis]{Yang2020PhysicsInformedGA}
Yang, L., Zhang, D., and Karniadakis, G.~E.
\newblock Physics-informed generative adversarial networks for stochastic
  differential equations.
\newblock \emph{ArXiv}, abs/1811.02033, 2020.

\end{thebibliography}
\bibliographystyle{icml2022}
\end{document}